\documentclass{amsart}
\usepackage[margin=3.5cm]{geometry}
\usepackage{graphicx} 
\usepackage{amsmath} 
\usepackage{epstopdf} 
\usepackage{amssymb} 
\usepackage{amsthm}
\usepackage{mathtools}
\usepackage{amstext}
\usepackage{relsize}
\usepackage{fancyhdr}
\usepackage{array}
\usepackage{caption}
\usepackage{float}
\usepackage{graphicx,xcolor,colortbl}    
\usepackage{hyperref}  
\usepackage{cleveref}
\usepackage{xspace}
\usepackage[title]{appendix}
\usepackage[makeroom]{cancel}
\usepackage{soul}
\usepackage{multirow}
\usepackage{graphicx}
\usepackage{tikz}
\usetikzlibrary{arrows}
\usetikzlibrary{decorations.markings}
\usepackage{siunitx}
\usepackage{comment}
\usepackage[indent=20pt]{parskip}
\usepackage[makeroom]{cancel}
\usetikzlibrary{arrows.meta}

\newtheorem{assumption}{Assumption}[section]
\newtheorem{lemma}{Lemma}[section]
\newtheorem{definition}{Definition}[section]
\newtheorem{remark}{Remark}[section]
\newtheorem{theorem}{Theorem}[section]

\newcommand\numberthis{\addtocounter{equation}{1}\tag{\theequation}}

\newcommand{\hmin}{h_{\min}}
\newcommand{\hmax}{h_{\max}}
\newcommand{\ito}{It\^{o}\xspace}

\newcommand{\CxJ}{C_{\texttt{X}}\xspace}

\newcommand{\Czeta}[1]{\bar L}

\newcommand{\order}{\delta}

\newcommand{\GamRX}{\Gamma_{2,n}}
\newcommand{\GamR}{\Gamma_2}

\newcommand{\Yjump}[1]{Y^{\texttt{\normalfont J}}(#1)}

\newcommand{\YjumpBefore}[1]{\widehat{Y}(#1)}

\newcommand{\YjumpStart}{\widetilde Y^{\texttt{J}}_0}

\newcommand{\Xjump}[1]{X^{\texttt{\normalfont J}}(#1)}

\newcommand{\XjumpBefore}[1]{\widehat{X}(#1)}

\newcommand{\XjumpStart}{X^{\texttt{J}}_0}

\newcommand{\hjump}[1]{h_{#1}}

\newcommand{\NumJumps}[1]{\bar N_{#1}}

\newcommand{\Ejump}[1]{E^{\texttt{\normalfont J}}(#1)}

\newcommand{\EjumpBefore}[1]{\widehat{E}(#1)}

\newcommand{\EThetaJump}[1]{E_{\mainMap}^{x,y}(#1)}

\newcommand{\EVarphiJump}[1]{E_{\bsMap}^{x,y}(#1)}

\newcommand{\Epsi}[3]{E_{\psi}^{#2,#3}(#1)}

\newcommand{\JAAM}{JA-AMs\xspace}
\newcommand{\JAAMM}{JA-AMM\xspace}

\newcommand{\mainMap}{\mathcal{M}}
\newcommand{\bsMap}{\mathcal{B}}
\newcommand{\finalC}[1]{C(\lambda,\rho,#1)}

\newcommand{\nth}[1]{${#1}^{\text{th}}$}

\usepackage[foot]{amsaddr}
\title[Adaptive time-stepping for SDEs with jumps]{Strong convergence of a class of adaptive numerical methods for SDEs with jumps}

\author{C\'onall Kelly$^{\,\star}$}
\address{$\star\,$ School of Mathematical Sciences, University College Cork, Ireland. \texttt{conall.kelly@ucc.ie}}

\author{Gabriel Lord$^{\,\dagger}$}
\address{$\dagger\,$ Mathematics, IMAPP, Radboud University, Nijmegen, The Netherlands. \texttt{gabriel.lord@ru.nl}} 

\author{Fandi Sun$^{\,\ddagger}$}
\address{$\ddagger\,$ Department of Mathematics, Heriot-Watt University, Edinburgh, UK. \texttt{fandi.sun@outlook.com}}

\date{\today}

\keywords{Stochastic Jump Differential Equations \and Adaptive Timestepping \and Jump-adapted Mesh \and Non-globally Lipschitz Coefficients \and Strong Convergence}

\begin{document}

\begin{abstract} 
We develop adaptive time-stepping strategies for It\^o-type stochastic differential equations (SDEs) with jump perturbations. Our approach builds on adaptive strategies for SDEs. Adaptive methods can ensure strong convergence of nonlinear SDEs with drift and diffusion coefficients that violate global Lipschitz bounds by adjusting the stepsize dynamically on each trajectory to prevent spurious growth that can lead to loss of convergence if it occurs with sufficiently high probability. In this article we demonstrate the use of a jump-adapted mesh that incorporates jump times into the adaptive time-stepping strategy. We prove that any adaptive scheme satisfying a particular mean-square consistency bound for a nonlinear SDE in the non-jump case may be extended to a strongly convergent scheme in the Poisson jump case where jump and diffusion perturbations are mutually independent and the jump coefficient satisfies a global Lipschitz condition. 
\end{abstract} 

\maketitle


\section{Introduction}
We investigate the use of adaptive time-stepping strategies in the construction of a class of strongly convergent numerical schemes for a $d$-dimensional stochastic jump differential equation (SJDE) shown as
\begin{multline}
X^J(t)= X^J_0+\int_{0}^{t}f\big(X^J(r-)\big)dr +\sum_{i=1}^{m}\int_{0}^{t}g_i\big(X^J(r-)\big)dW_i(r)\\
+\int_{0}^{t}\int_Z\gamma\big(z,X^J(r^-)\big) J_{\nu}(dz\times dr),\label{eq:true_jump}
\end{multline}
where $t\in[0,T]$, $T\geq 0$ and $W=[W_1,\cdots, W_m]^T$ is an $m$-dimensional Wiener process. In \eqref{eq:true_jump}, the integral with respect to $W$ is of \ito type, and the jump contribution is expressed as an integral with respect to a Poisson random measure $J_\nu$ with finite intensity measure $\nu$.

Let solutions of the SJDEs \eqref{eq:true_jump} and its jump-free version \eqref{eq:true} be defined on a filtered probability space $(\Omega , \mathcal{F} , (\mathcal{F}_t )_{t \geq 0} , \mathbb{P})$. The initial vector $X^J(0)=X^J_0\in L_2(\mathbb{R}^d)$ is $\mathcal{F}_0$-measurable. We consider the case where the filtration $(\mathcal{F}_t)_{t\geq 0}$ admits the decomposition $\mathcal{F}_t=\sigma\big(\mathcal{G}_t\cup \mathcal{H}_T\big)$ for all $t\geq 0$, where $W$ is adapted to the filtration $(\mathcal{G}_t)_{t\geq 0}$, and $J_\nu$ is adapted to the filtration $(\mathcal{H}_t)_{t\geq 0}$, and the $\sigma$-algebra $\mathcal{G}_t$ is independent of $\mathcal{H}_T$ for all $t\in[0,T]$. This will be the case if $W$ and $J_\nu$ are independent, and allows for the jump times to be pre-computed at time $t=0$.

The drift coefficient $f: \mathbb{R}^d \rightarrow \mathbb{R}^d$ and the diffusion coefficient $g: \mathbb{R}^d\rightarrow \mathbb{R}^{d\times m}$ are both twice continuously differentiable and each satisfy a local Lipschitz condition along with a polynomial growth condition and, together, a monotone condition. The jump coefficient $\gamma:Z\times\mathbb{R}^d\rightarrow\mathbb{R}^d$ satisfies a global Lipschitz condition in the second argument. When $\gamma\equiv 0$, \eqref{eq:true_jump} becomes the $d$-dimensional SDE of \ito type 
\begin{align}
X(t)= X(0)+\int_{0}^{t}f(X(r))dr +\sum_{i=1}^{m}\int_{0}^{t}g_i(X(r))dW_i(r)\label{eq:true}.
\end{align}

SJDEs are used to model processes subject to perturbation both by continuous diffusion and event-driven random effects. For example asset price models in finance are subject to the continuous effects of market volatility as well as discrete shocks such as market crashes. Similarly, in ecology the dynamics of an ecosystem will be influenced by ongoing random effects due to variations in (for example) food supply and fecundity, but also by sudden environmental inputs such as fertiliser or toxin spills.

Strategies for the numerical approximation of SDEs with jumps such as \eqref{eq:true_jump} can be divided broadly into two categories: jump adapted and non-jump adapted. Jump adapted approaches discretise the interval $[0,T]$ so that all jump times are included as meshpoints. Non-jump adapted approaches do not do this, and allow jump times to fall between meshpoints. Jump adapted discretization for jump SDEs with globally Lipschitz $f$, $g_i$ and $\gamma$ was introduced in \cite{bruti2007strong}. The approach allows any strongly convergent numerical method for the non-jump SDE \eqref{eq:true} to be extended to the SJDE \eqref{eq:true_jump} via a mesh that is a superposition of jump times and equidistant time steps. We are motivated by \cite{bruti2007strong} to introduce Jump Adapted-Adaptive Methods (\JAAM) in this article. 

Our intention is to handle nonlinear functional response from the drift and diffusion coefficients $f$ and $g_i$ by the kind of procedural adaptive time-stepping used in \cite{KLSBit}, but to also ensure that jump times are included as meshpoints. At each step, either the next adaptive landing point or the next jump time is taken, whichever is sooner. Our use of adaptive time-stepping rather than an equidistant mesh to handle the nonlinearity of the coefficients allows us to weaken the assumptions on $f$ and $g_i$ so that they do not need to be globally Lipschitz continuous. This approach allows us to build upon any adaptive numerical method that satisfies certain local error conditions for the non-jump SDE \eqref{eq:true} to achieve strong $L_2$ convergence for the SJDE \eqref{eq:true_jump}. In particular we can construct a Jump Adapted-Adaptive Milstein Method (\JAAMM) that builds upon analysis of the adaptive Milstein method for \eqref{eq:true} in \cite{KLSBit} to achieve order-one strong $L_2$ convergence for \eqref{eq:true_jump} even in the case where the diffusion coefficient $g$ is not required to satisfy a commutativity condition.

Since the jump adapted approach can be computationally expensive when used to simulate processes with high jump intensity, there has also been recent interest in extending the applicability of numerical schemes developed for nonlinear SDEs to SJDEs using the non-jump adapted approach. We summarise some of the most important works here and indicate to what class of coefficient each method applies. For SDEs with globally Lipschitz $f$, $g_i$ and $\gamma$, Higham \& Kloeden~\cite{higham2006convergence} treated an implicit Euler-Maruyama method and Ren \& Tian~\cite{ren2022mean} introduced the two-step Milstein method. For the case where $\gamma$ remains globally Lipschitz continuous, but $f$ and each $g_i$ together satisfy a monotone condition, Deng et al~\cite{deng2019truncated} introduced the truncated Euler-Maruyama method. For SDEs where all of $f$, $g_i$ and $\gamma$ together satisfy a monotone condition, Dareiotis et al~\cite{dareiotis2016tamed} introduced the tamed Euler-Maruyama method, Li et al~\cite{li2021compensated} introduced the compensated projected Euler-Maruyama method, and Kumar \& Sabanis~\cite{kumar2014tamed} introduced the tamed Milstein method.

Adapative time-stepping to control numerical growth has been shown to be effective for SDEs as an alternative to taming (e.g. \cite{hutzenthaler2011strong,kumar2014tamed}) and projection-type methods (e.g. \cite{beyn2017stochastic,li2021compensated}), which are all (in general) more computationally efficient than implicit methods. See for example \cite{kelly2018adaptive,fang2016adaptive,reisinger2020adaptive}. The adaptive approach provides automatic stepsize reduction in domains where the SDE is locally stiff. This article extends work contained in the PhD thesis of the third author~\cite{Fandithesis}, which contains a detailed  analysis of an adaptive Milstein scheme for SJDEs.

The structure of the article is as follows. Mathematical preliminaries are given in Section \ref{sec:prelim}, including precise specifications of the conditions imposed on each $f$, $g_i$ and $\gamma$. The construction of \JAAM is shown in Section \ref{sec:adpt_strategy_SDE_jump}. The main result, on strong $L_2$ convergence of \JAAM of order 1, is in Section \ref{sec:result_jump}. Finally, numerical examples are in Section \ref{sec:numerics}.

\section{Mathematical preliminaries}\label{sec:math_intro_jump}\label{sec:prelim}
For all $x \in \mathbb{R}^d$ and for all $\phi(x)\in \mathrm{C}^2(\mathbb{R}^d,\mathbb{R}^d)$, the Jacobian matrix of $\phi(x)$ is denoted $\mathbf{D}\phi(x)\in\mathcal{L}(\mathbb{R}^d,\mathbb{R}^d)$; the second derivative of $\phi(x)$ with respect to a vector $x$ forms a 3-tensor and is denoted $\mathbf{D}^2\phi(x)\in\mathcal{L}(\mathbb{R}^{d\times d},\mathbb{R}^d)$; and $[x]^2:=x\otimes x$ stands for the outer product of $x$ and itself.
Furthermore, let $\|\cdot\|$ denote the standard Euclidean  norm in $\mathbb{R}^d$, $\|\cdot\|_{\mathbf{F}(a\times b)}$ the Frobenious norm in $\mathbb{R}^{a\times b}$; for simplicity we write $\|\cdot\|_{\mathbf{F}}$ as the Frobenious norm in $\mathbb{R}^{d\times d}$. $\|\cdot\|_{\mathbf{T}_3}$ denotes the induced tensor norm (spectral norm) in $\mathbb{R}^{d\times d\times d}$ and it is defined as $\big\|\cdot \big\|_{\mathbf{T}_3}:=\sup_{h_1,h_2\in\mathbb{R}^d, \|h_1\|,\|h_2\|\leq 1}\big\|\cdot(h_1\otimes h_2)\big\|$. We denote the left-hand limit as $\phi(t^-):=\lim_{s\rightarrow t,s<t}\phi(s)$.
Finally, for $a,b\in\mathbb{R}$, $a\vee b$ denotes max$\{a,b\}$, and $a\wedge b$ denotes min$\{a,b\}$. 

\subsection{Finite-activity Poisson jump processes}
Let $(\pi_i)_{i\geq 1}$ be a sequence of independent and identically distributed (i.i.d.) exponential random variables with  parameter $\lambda\in\mathbb{R}$ and define 
\begin{equation}\label{eq:taui}
\tau_i:=\sum_{j=1}^{i}\pi_j.
\end{equation}
We will interpret the sequence of increasing non-negative random variables $(\tau_i)_{i\geq 1}$ as the jump times of a Poisson process where $\lambda$ is the jump intensity, so that each $\pi_i \sim \exp(1/\lambda)$ represents the waiting time between two jumps.

\begin{definition}\label{def:N_bar}
The homogeneous Poisson point process $(\NumJumps{t})_{t\geq 0}$ counts the number of jumps that take place over the interval $[0,t]$ for all $t\geq 0$, written
$$
\NumJumps{t}:= \# \big\{i\geq 1,\,\, \tau_i\in [0,t]\big\}.
$$ 
Now define the Poisson random measure with finite intensity measure $\nu$ as $J_\nu:Z\times [0,T] \rightarrow \mathbb{N}$ where with $Z \subseteq \mathbb{R}^d\backslash\{0\}$ represents the domain of possible jump sizes
\begin{equation}\label{eq:Jnu}
    J_\nu(dz \times dt) := \# \Big\{i\geq 1,\,\, (\zeta_i, \tau_i )\in dz \times  dt \Big\},
\end{equation}
and $dz$, $dt$ are Lebesgue-measurable subsets of $Z$, $[0,T]$ respectively. 
\end{definition}
\begin{remark}
\label{rem:J(s)_define}
In \eqref{eq:Jnu}, $J_\nu(dz\times dt)$ counts the number of jumps in $dt\subseteq (0,T]$ with jump size contained in the set $dz\subseteq Z$. For $i\geq 1$, each $\zeta_i$ represents the size of the jumps observed at the corresponding jump time $\tau_i$. 
\end{remark}
The intensity measure $\nu$ is characterised by 
$$
\nu(dz)m(dt):=\mathbb{E}[J_{\nu}(dz\times dt)],
$$ 
where $m(dt)$ is the Lebesgue measure of $dt$. By setting $dz=Z$ and $dt=[0,1]$ we can see that $\nu(Z)=\lambda<\infty$.

\begin{definition}
Given a jump coefficient $\gamma: Z \times \mathbb{R}^d \rightarrow \mathbb{R}^d$  we can construct a jump process $\chi$ with respect to the Poisson random measure $J_\nu$ as 
\begin{equation}
    \chi(t)=\int_0^t\int_Z\gamma\big(z,u(s)\big)J_\nu(dz\times ds)=\sum_{\{i,\,\tau_i\in[0,t]\}}\gamma\big(\zeta_i, u(\tau_i)\big). \label{eq:jump_sum=int}
\end{equation}
for any c\`adl\`ag process $u:[0,T]\to\mathbb{R}^d$ adapted to the natural filtration of $J_\nu$.
\end{definition}
Since we will be working on a jump adapted mesh which requires analysis at exact jump times,  the following convention is useful:
\begin{multline} \label{eq:def_measureAtPoint}
    \int_Z\gamma(z, u^-)J_{\nu}(dz\times\{t\}):=\int_0^t\int_Z\gamma\big(z,u(s)\big)J_{\nu}(dz\times ds)-\int_0^{t^-}\int_Z\gamma\big(z,u(s)\big)J_{\nu}(dz\times ds). 
\end{multline}
 
\subsection{Existence and uniqueness of solutions}

Our analysis extends the strong convergence of a numerical method for \eqref{eq:true} to the jump case \eqref{eq:true_jump}, so we present assumptions for each in turn.

\begin{assumption} \label{ass:f+g}
Let $f\in \mathrm{C}^2(\mathbb{R}^d,\mathbb{R}^d)$ and $g\in\mathrm{C}^2(\mathbb{R}^d,\mathbb{R}^{d\times m})$ with each column of $g$ that $g_i(x)=[g_{1,i}(x),\dots,g_{d,i}(x)]^T\in\mathrm{C}^2(\mathbb{R}^d,\mathbb{R}^{d})$. For each $\varkappa \geq 1$  there exist $L^{\varkappa}>0$ such that
\begin{align}
    \big\|f(x)-f(y)\big\|^2+\big\|g(x)-g(y)\big\|^2_{\mathbf{F}(d\times m)}\leq L^{\varkappa}\big\|x-y\big\|^2,  \label{eq:local_lipschitz}
\end{align}
for $x,y\in\mathbb{R}^d$ with $\|x\| \vee \|y\|\leq \varkappa$, and there exists $L\geq 0$ such that for some $\eta \geq 2$
\begin{align}
    \big\langle x-y, f(x)-f(y)\big\rangle+\frac{\eta-1}{2} \big\|g(x)-g(y)\big\|^2_{\mathbf{F}(d\times m)} \leq L\big\|x-y\big\|^2.  \label{eq:monotone}
\end{align}
\end{assumption}
Under \eqref{eq:local_lipschitz} and \eqref{eq:monotone}, the SDE \eqref{eq:true} has a unique strong solution on any interval $[0, T ]$, where $T < \infty$ on the filtered probability space $(\Omega,\mathcal{F},(\mathcal{F}_t)_{t\geq 0},\mathbb{P})$, see~\cite{Hasminskii}, \cite{mao2007SDEapp} and \cite{tretyakov2013fundamental}. 

Now additionally assume for jump coefficient $\gamma$ the following.
\begin{assumption}\label{ass:jump_gamma}
Suppose that $f\in \mathrm{C}^2(\mathbb{R}^d,\mathbb{R}^d)$ and  $g\in\mathrm{C}^2(\mathbb{R}^d,\mathbb{R}^{d\times m})$ with each $g_i(x)=[g_{1,i}(x),\dots,g_{d,i}(x)]^T\in\mathrm{C}^2(\mathbb{R}^d,\mathbb{R}^{d})$ for $i=1,\dots,d$. Let $\gamma\in \mathrm{C}^2\big((\mathbb{R}^d \backslash \{0\})\times\mathbb{R}^d,\mathbb{R}^d\big)$. Let $x,y\in\mathbb{R}^d$ with $\|x\|\vee\|y\|\leq\varkappa$, and for $\varkappa>1$, there exists $L^{\varkappa} \leq \infty$ such that 
\begin{multline}
    \quad 2\big\langle x-y, f(x)-f(y)\big\rangle+ \big\|g(x)-g(y)\big\|^2_{\mathbf{F}(d\times m)}\\
    +\int_Z\big\|\gamma(z,x)-\gamma(z,y)\big\|^2\nu(dz) \,\,\leq\,\, L^{\varkappa}\big\|x-y\big\|^2.  \label{eq:local_lipschitz_jump}
\end{multline}
Furthermore, we assume that there exist $\Czeta{1}$ such that
\begin{align}
    \int_Z\big\|\gamma(z,x)-\gamma(z,y)\big\|^2\nu(dz) \leq&\,\, \Czeta{1}\big\|x-y\big\|^2.  \label{eq:monotone_gamma}
\end{align}
\end{assumption}

Under Assumptions \ref{ass:f+g} and \ref{ass:jump_gamma}, the SJDE \eqref{eq:true_jump} has a unique strong solution on any interval $[0, T ]$, where $T < \infty$, see~\cite[Thm. 2]{gyongy1980stochastic}.

\begin{definition}\label{def:stoppingTime}
A random variable $\mu:\Omega\rightarrow [0,\infty]$ is called an $\mathcal{F}_t$-stopping time (or simply, stopping time) if $\{\omega:\mu(\omega)\leq t\}\in \mathcal{F}_t$ for any $t\geq 0$. If $\mu$ is a stopping time, define 
\begin{align}
   \mathcal{F}_{\mu}:=\{A\in\mathcal{F}\,:\,A\cap\{\omega:\mu(\omega)\leq t\}\in\mathcal{F}_t,\,\text{ for all }\,t\geq 0\},  \label{def:filtration_stopping}
\end{align}
which is a sub-$\sigma$-algebra of $\mathcal{F}$.
\end{definition}

\section{Jump Adapted-Adaptive Methods (\JAAM)}\label{sec:adpt_strategy_SDE_jump}

\subsection{Jump adapted meshes with procedurally generated nodes}
Let $\{h_{n+1}\}_{n\in\mathbb{N}}$ be a sequence of strictly positive random timesteps with corresponding random times $\{t_n:=\sum_{i=1}^{n}h_i\}_{n\in\mathbb{N}\backslash \{0\}}$, where $t_0=0$.

\begin{lemma}\label{lem:stoppingTimes}
For each $k\in\mathbb{N}$, if $t_k$ is an $\mathcal{F}_t$-stopping time and  $h_{k+1}$ is $\mathcal{F}_{t_k}$-measurable, then $t_{k+1}$ is an $\mathcal{F}_t$-stopping time.
\end{lemma}
\begin{proof}
By assumption in the statement of the Lemma, $\{t_k \leq t\}\in \mathcal{F}_t$ for all $t\in [0,T]$. By \eqref{def:filtration_stopping} in Definition \ref{def:stoppingTime} we can define the filtration at the stopping time $t_k$ as the $\sigma$-algebra 
\begin{equation}\label{eq:Ftk}
\mathcal{F}_{t_k}:=\{A\in\mathcal{F}:A \cap \{ t_k\leq t\}\in \mathcal{F}_t\text{ for all }t\in[0,T]\}.
\end{equation}
Since $t_n=\sum_{i=1}^{n}h_i$ for $n\in\mathbb{N}\backslash \{0\}$, we have
    \begin{align*}
        \{t_{k+1}\leq t\}=\{h_{k+1}+t_k\leq t\}=\bigcup\limits_{r\in[0,t]\cap\mathbb{Q}}
        \Big(
        \{h_{k+1}\leq t-r\} \cap \{t_k\leq r \} \Big).
    \end{align*}
If we now assume that $h_{k+1}$ is an $\mathcal{F}_{t_k}$-measurable random variable and set $A=\{h_{k+1}\leq t-r\} $, then by \eqref{eq:Ftk} we have $A\cap\{t_k\leq r\}\in\mathcal{F}_r$ which is a sub-$\sigma$-algebra of $\mathcal{F}_t $ for all $r\in[0,T]$. Therefore, $\{t_{k+1}\leq t\}\in\mathcal{F}_t $ for all $t\in[0,T]$.
\end{proof}
\begin{remark}\label{rem:hstopp}
It is a consequence of Lemma \ref{lem:stoppingTimes} that, as long as $h_1$ can be computed from the initial data, and successive adaptive timesteps $h_{k+1}$ are computed on each trajectory as a deterministic function of the numerical approximation at time $t_k$, then every node on the random grid thus constructed is an $\mathcal{F}_t$-stopping-time. 
\end{remark}

\begin{assumption} \label{ass:h}
For the sequence of random timesteps $\{h_{n+1}\}_{n\in\mathbb{N}}$, there are constant values $\hmax>h_{\min}>0$, $\rho>1$ such that $h_{\max}=\rho  h_{\min}$, and
\begin{equation}\label{eq:hmaxhmin}
0<h_{\min} \leq h_{n+1} \leq h_{\max}\leq 1.
\end{equation}
In addition, we assume each $h_{n+1}$ is $\mathcal{F}_{t_n}$-measurable.  
\end{assumption}
The second statement in Assumption \ref{ass:h} on measurability allows us to condition on $\mathcal{F}_{t_n}$ at any point on the random time-set $\{t_n\}_{n\in\mathbb{N}}$, and its validity is a consequence of Remark \ref{rem:hstopp}.

\begin{definition}\label{def:N}
Let $N^{(t)}$ be a random integer such that 
 \begin{align}
 N^{(t)}:=\max\{n\in\mathbb{N}\backslash \{ 0\}: t_{n-1}<t \},   \label{eq:def_N}
 \end{align}
and let $N=N^{(T)}$ and $t_N=T$, so that $T$ is always the last point on the mesh. 
\end{definition}

Each jump time $\tau_{n+1}$ (defined in \eqref{eq:taui}) is $\mathcal{H}_{T}$-measurable and therefore also $\mathcal{F}_{\tau_n}$-measurable. This is ensured if all jump times and sizes can be precomputed, and will be the case when the diffusion and jump processes are independent, and the jump intensity $\lambda$ is constant. However it is not necessary to actually to implement the method in this way.

Let Definition \ref{def:N} be satisfied with $N^{(t)}_{\min}:=\lfloor t/\hmax \rfloor$ and  $N^{(t)}_{\max}:=\lceil t/\hmin + J_\nu(Z,(0,t])\rceil$ where $J_\nu(Z,(0,t])<\infty$ a.s. is the number of jumps of any size on the interval $(0,t]$ (see Remark \ref{rem:J(s)_define}). Notice that $N^{(t)}_{\max}$ is also an $(\mathcal{F}_t)_{t\geq 0}$-measurable random variable, which a major challenge in the proof of Theorem \ref{thm:result_jump} that is specific to the jump setting. The lower bound $\hmin$ given by \eqref{eq:hmaxhmin}, along with the a.s. finiteness of the number of jumps, ensures that a simulation over the interval $[0,T]$ can be completed in a finite number of time steps.

We ensure that we reach the final time by taking $h_{N}=T-t_{N-1}$ as our final step, and use the backstop method if $h_{N}<\hmin$. In the absence of jumps (i.e. if $\gamma\equiv 0$) the last step is the only occasion of the backstop method taking on a step that might be smaller than $h_{\min}$. Otherwise this can occur at any step during the process, because according to the jump adapted adaptive time-stepping strategy in \eqref{eq:defh_jump} the lower bound of each step is the distance to the next jump time rather than the $\hmin$ that we choose, which might be smaller than $\hmin$. 

\subsection{The class of Jump Adapted-Adaptive Methods (\JAAM)}
Now we introduce a continuous version of the \JAAM class of schemes characterised by the sequence of tuples $\big\{\big(\Yjump{s}\big)_{s\in[t_n,t_{n+1}]}, \hjump{n+1} \big\}_{n\in\mathbb{N}}$ with individual terms in the sequence composed of the continuous version of the approximate solution over each time-step along with the length of the corresponding timestep. The corresponding discrete approximation is given by setting $s=t_n$ for each term in the sequence.
\subsubsection{The main and backstop maps}
Our approach is a hybrid of two methods, each associated with maps $\mainMap:\mathbb{R}^d\times[0,T]\times[0,h_{\max}]$ and $\bsMap:\mathbb{R}^d\times[0,T]\times[0,h_{\max}]$. We will call $\mainMap$ the main, or default, map and $\bsMap$ the backstop map. The hybridisation of two maps is motivated by the adaptive approach proposed in \cite{kelly2018adaptive}, where  $\mainMap$ corresponds to an efficient scheme which is convergent as an adaptive method but not as a fixed-step method and $\bsMap$ corresponds to a method which is convergent as a fixed step method but which may be inefficient or may induce distortions in solution dynamics. Our aim is that it should be used only rarely: this is the subject of Theorem \ref{thm:backstopProb}. Examples of potential maps $\mainMap$ and $\bsMap$ are given in Section \ref{sec:Mil} and implemented in Section \ref{sec:numerics}.

We characterise $\mainMap$ and $\bsMap$ through their action over a single step, requiring them to satisfy a certain mean-square consistency bound in the absence of jumps. 
\begin{definition}
For $s\in[t_n,t_{n+1}]$, $n\in\mathbb{N}$ and any $x,y\in\mathbb{R}^d$, let
\begin{eqnarray}
    \EThetaJump{s}&:=& \Big( x -  \mainMap\left(y\boldsymbol{,}\,\,t_n\boldsymbol{,}\,\, s -t_n \right) \Big)\mathbf{1}_{\{h_{\min}<\hjump{n+1}\leq h_{\max}\}}\label{eq:defE(r)_jump}\\
\EVarphiJump{s}&:=&\Big( x - \bsMap\left(y\boldsymbol{,}\,\,t_n\boldsymbol{,}\,\, s -t_n \right)\Big)\mathbf{1}_{\{\hjump{n+1}\leq h_{\min}\}}.\nonumber 
\end{eqnarray}
\end{definition}
Then we impose the following assumption:

\begin{assumption}\label{ass:one-stepBounds}
We suppose that $\mainMap,\bsMap :\mathbb{R}^d\times[0,T]\times[0,h_{\max}] \rightarrow \mathbb{R}^d$ satisfy, for some $\order>0$, both $\psi=\mainMap$ or $\psi=\bsMap$, and any a.s. finite $\mathcal{F}_{t_n}$-measurable $\mathbb{R}^d$-valued random variables  $A_n,B_n$,
\begin{multline}
    \mathbb{E}\Big[\Big\|  \Epsi{t_{n+1}}{A_n}{B_n}\Big\|^2 \Big|\mathcal{F}_{t_n}\Big] \leq \|A_n-B_n\|^2 \\
    +\Gamma_{1}\int_{t_n}^{t_{n+1}} \mathbb{E}\left[\big\|\Epsi{r}{A_n}{B_n}\big\|^2 \middle|\mathcal{F}_{t_n}\right]dr+\GamRX\,{h_{n+1}^{2\order+1}}, \quad a.s.  \label{eq:Case_I_twostep}
\end{multline} 
where $\Gamma_{1}<\infty$ is constant, $\GamRX$ is a scalar $\mathcal{F}_{t_n}$-measurable random variable with finite expectation denoted $\GamR$, and both $\Gamma_1$ and $\GamRX$ are independent of $\hmax$. 
\end{assumption}
\begin{remark}
This bound describes the effect of the application over a single step of either of the non-jump maps $\mathcal{M}$ or $\mathcal{B}$ to the solution of the SJDE at the beginning of the step $X^J(t_n)$. In the (non-jump) SDE case \eqref{eq:true} we would have $A_n = X(t_n)$ and $B_n = X_n$ where, for example, $X_n$ is a Milstein approximation. Thus \eqref{eq:Case_I_twostep} would correspond to a local error estimate (similar to those in \cite{tretyakov2013fundamental} and proved in \cite{KLSBit}). In our setting we wish to exploit this kind of non-jump SDE estimate where the initial inputs from the exact and approximate processes come from the SJDE but the estimate \eqref{eq:Case_I_twostep} can nonetheless be checked in the SDE setting. 
Hence we introduce the random variables $A_n$ (corresponding to $\Xjump{t_n}$) and $B_n$ (corresponding to the numerical approximation of the SJDE at time $t_n$).
\end{remark}
For convenience we rewrite \eqref{eq:true_jump} as for $s\in[t_n, t_{n+1}]$,  $n\in\mathbb{N}$
\begin{align}
\XjumpBefore{s}:=& \Xjump{t_n}+\int_{t_n}^{s}f(\Xjump{r})dr +\sum_{i=1}^{m}\int_{t_n}^{s}g_i(\Xjump{r})dW_i(r);\label{eq:Xjump_true_1}\\
\Xjump{s}=& \XjumpBefore{s^-}+ \int_Z\gamma\big(z,\XjumpBefore{s^-}\big) J_{\nu}(dz\times \{s\}). \label{eq:Xjump_true_2}
\end{align}

We may now define our hybrid scheme in terms of $\mainMap,\bsMap$, and their action over each step before and after a jump.

\begin{definition}[\JAAM Scheme] \label{def:j-adapted adaptive explicit Milstein scheme}
Let  $\{h_{n+1}\}_{n\in \mathbb{N}}$ satisfy Assumption \ref{ass:h}. 
We define the continuous form of a \emph{\JAAM scheme} associated with  $\{h_{n+1}\}_{n\in\mathbb{N}}$
\begin{align} 
\YjumpBefore{s}:=\,\,&\mainMap\left(\Yjump{t_n} \boldsymbol{,}\,\,  t_n\boldsymbol{,}\,\,s-t_n\right) \cdot \mathbf{1}_{\{h_{\min}<\hjump{n+1} \leq h_{\max}\}}\label{eq:AT_twoStep_1}\\
&\quad+\bsMap\left(\Yjump{t_n} \boldsymbol{,}\,\,t_n\boldsymbol{,}\,\,s-t_n\right)  \cdot \mathbf{1}_{\{\hjump{n+1} \leq h_{\min}\}},\nonumber \\
\Yjump{s}:=\,\,&\YjumpBefore{s^-}+\int_{Z}\gamma\big(z,\YjumpBefore{s^-}\big)J_{\nu}(dz\times\{s\}) 
\label{eq:AT_twoStep}
\end{align}
for $s\in[t_n,t_{n+1}]$, $n\in\mathbb{N}$. Here, $\Yjump{0}=X^J(0)$, and
\begin{equation}\label{eq:jump_indicator}
\int_{Z}\gamma\big(z,\YjumpBefore{s^-}\big)J_{\nu}(dz\times\{s\})=\left\{\begin{array}{ll}
\gamma\left(\zeta_{\NumJumps{s}},\YjumpBefore{s^-}\right),& s=\tau_{\NumJumps{s}};\\
0,& s\neq \tau_{\NumJumps{s}}.
\end{array}\right.
\end{equation}
\end{definition}

Now we define notation for the global error of the scheme. 
\begin{definition}
For any $s\in[0,T]$, we distinguish between the error acquired by \JAAM up to but not including a possible jump time
\begin{equation}\label{eq:defE(r)_combined_1} 
\EjumpBefore{s}:=\XjumpBefore{s^-}-\YjumpBefore{s^-};
\end{equation}
and at a possible jump time $s$
\begin{align}
   \Ejump{s}:=\,& \Xjump{s}-\Yjump{s} \nonumber\\
   =\,&\EjumpBefore{s}+\int_Z \Delta \gamma \Big(z, \XjumpBefore{s^-}, \YjumpBefore{s^-} \Big)J_{\nu}(dz\times\{s\}), \label{eq:defE(r)_combined_jump}
\end{align}
where 
\begin{align}\label{eq:errorGamma}
    \Delta \gamma \Big(z,\XjumpBefore{s^-}, \YjumpBefore{s^-} \Big):=\,&\gamma \Big(z,\XjumpBefore{s^-} \Big) -\gamma \Big(z,\YjumpBefore{s^-}\Big). 
\end{align}
\end{definition}
Notice that $J_{\nu}(dz\times\{s\})=0$ when $s$ is not a jump time.

\section{Strong convergence of \JAAM}
\label{sec:result_jump}
For any adaptive numerical scheme that satisfies local error estimates of the form given in Assumption \ref{ass:one-stepBounds} for the non-jump SDE \eqref{eq:true}, our main result demonstrates strong $L_2$ convergence when extended to the jump SDE \eqref{eq:true_jump} via a jump adapted mesh. We give the adaptive Milstein method as an example of $\mainMap$ in Section \ref{sec:Mil}, leading to a specific Jump Adapted-Adaptive Milstein Method (\JAAMM) scheme of order $\delta=1$. 
\begin{theorem}[Strong Convergence] \label{thm:result_jump}
Let $(\Xjump{t})_{t\in[0,T]}$ be a solution of \eqref{eq:true_jump} with initial value $\Xjump{0}= \XjumpStart\in\mathbb{R}^d$ and that Assumptions \ref{ass:f+g} and \ref{ass:jump_gamma} hold. 

Let $\big\{\big(\Yjump{s}\big)_{s\in[t_n,t_{n+1}]},\hjump{n+1}\big\}_{n\in\mathbb{N}}$ be a numerical scheme from the \JAAM class as characterised in Definition \ref{def:j-adapted adaptive explicit Milstein scheme} with $\YjumpStart = \XjumpStart$ such that Assumption \ref{ass:one-stepBounds} holds with some $\order>0$. 
Then there exists a constant $\finalC{T} > 0$ such that 
\begin{equation*} 
\max_{t\in[0,T]}\Big(\mathbb{E}\Big[\|\Xjump{t}-\Yjump{t}\|^2\Big]\Big)^{1/2} \leq \finalC{T}\,\hmax^\order.
\end{equation*}
\end{theorem}

The proof is best approached in five distinct steps. Step 1 reiterates a one-step error bound that applies over any step where no jump has taken place. Step 2 aggregates these one-step errors over a block of steps between two jumps, including the steps from the initial time to the first jump. Step 3 then aggregates the errors from Step 2 to provide an error estimate that applies from the initial time $t_0$ to the time of the last jump. Step 4 provide an error from the time of the last jump (at $\tau_{\bar N_{t}}$) to any time $t\in(\tau_{\bar N_{t}},T]$. Step 5 combines the error estimates from Steps 3 and 4. An application of the continuous Gr\"onwall inequality gives the statement of Theorem \ref{thm:result_jump}.

\begin{proof}
\textbf{\textit{Step 1 --  One-step error bound.}} 
By \eqref{eq:Case_I_twostep} in Assumption \ref{ass:one-stepBounds}, and setting $A_n=\XjumpBefore{t_n}$ and $B_n=\YjumpBefore{t_n}$, by \eqref{eq:defE(r)_combined_1} we have the one-step error bound prior to the inclusion of any jump, regardless of whether the main map $\mathcal{M}$ or the backstop map $\mathcal{B}$ has been used (see Definition \ref{def:j-adapted adaptive explicit Milstein scheme}) we have 
\begin{align}\label{eq:combine_error_noJ}
\mathbb{E}\Big[\big\|\EjumpBefore{t_{n+1}}\big\|^2 \Big|\mathcal{F}_{t_n}\Big] \leq \big\|\EjumpBefore{t_n}\big\|^2+ \Gamma_1\int_{t_n}^{t_{n+1}} \mathbb{E}\Big[\big\|  \EjumpBefore{r}\big\|^2\Big|\mathcal{F}_{t_n} \Big]dr +\GamRX{\hmax^{2\order+1}}. 
\end{align}

\textbf{\textit{Step 2 -- Jump-to-jump error bound.}} In this step, we calculate the error bound from \nth{k} jump to \nth{(k+1)} jump, that is from time $\tau_{k}$ to $\tau_{k+1}$, for $k\in[0,\NumJumps{t}]$. Firstly, by \eqref{eq:defE(r)_combined_1} and \eqref{eq:defE(r)_combined_jump} we have
\begin{equation}\label{eq:J2J_start}
    \Ejump{\tau_{k+1}}=\EjumpBefore{\tau_{k+1}}+\int_Z \Delta \gamma \Big(z,\XjumpBefore{\tau_{k+1}},\YjumpBefore{\tau_{k+1}} \Big)J_{\nu}(dz\times\{\tau_{k+1}\}).
\end{equation}
Since $\tau_{k+1}$ is a jump time, by \eqref{eq:jump_indicator} we have 
\begin{equation*}
    \int_Z \Delta \gamma \Big(z,\XjumpBefore{\tau_{k+1}^-},\YjumpBefore{\tau_{k+1}^-} \Big)J_{\nu}(dz\times\{\tau_{k+1}\}) =  \Delta \gamma \Big(\zeta_{k+1},\XjumpBefore{\tau_{k+1}^-},\YjumpBefore{\tau_{k+1}^-} \Big).
\end{equation*}
Taking norm squared on the both sides of \eqref{eq:J2J_start} followed by an application of Jensen's inequality we have
\begin{equation}\label{eq:E[tau_k+1]}
    \left\|\Ejump{\tau_{k+1}}\right\|^2\leq 2\left\|\EjumpBefore{\tau_{k+1}}\right\|^2+2\left\|\Delta \gamma \Big(\zeta_{k+1},\XjumpBefore{\tau_{k+1}^-},\YjumpBefore{\tau_{k+1}^-} \Big)\right\|^2.
\end{equation}
Taking expectation on the both sides of \eqref{eq:E[tau_k+1]} conditioned on $\mathcal{F}_{t_{N^{(\tau_{k+1})}-1}}$, and by \eqref{eq:monotone_gamma} we have 
\begin{multline*}
    \mathbb{E}\left[\left\|\Ejump{\tau_{k+1}}\right\|^2\middle| \mathcal{F}_{t_{N^{(\tau_{k+1})}-1}}\right]\leq 2\mathbb{E}\left[\left\|\EjumpBefore{\tau_{k+1}}\right\|^2\middle| \mathcal{F}_{t_{N^{(\tau_{k+1})}-1}}\right]
\\+2\mathbb{E}\left[\left\|\Delta \gamma \Big(\zeta_{k+1},\XjumpBefore{\tau_{k+1}^-},\YjumpBefore{\tau_{k+1}^-} \Big)\right\|^2\middle| \mathcal{F}_{t_{N^{(\tau_{k+1})}-1}}\right]. 
\end{multline*}
For the 2nd term on the RHS, by \eqref{eq:monotone_gamma} we have
\begin{multline}\label{eq:before_BIT}
    \mathbb{E}\left[\left\|\Ejump{\tau_{k+1}}\right\|^2\middle| \mathcal{F}_{t_{N^{(\tau_{k+1})}-1}}\right] \,\,\leq\,\,  \mathbb{E}\left[\left\|\EjumpBefore{\tau_{k+1}}\right\|^2\middle| \mathcal{F}_{t_{N^{(\tau_{k+1})}-1}}\right]
\\+(1+2\Czeta{1})\mathbb{E}\left[\left\|\EjumpBefore{\tau_{k+1}}\right\|^2\middle| \mathcal{F}_{t_{N^{(\tau_{k+1})}-1}}\right]. 
\end{multline}
By setting $t_{n}=t_{N^{(\tau_{k+1})}-1}$ and $t_{n+1} = \tau_{k+1}$ in \eqref{eq:combine_error_noJ} we have an error bound over the last step to the jump time $\tau_{k+1}$. 
Substituting back into \eqref{eq:before_BIT} to replace the 1st term on the RHS, we have
\begin{multline}\label{eq:combine_error_withJ}
\mathbb{E}\Big[\big\|\Ejump{\tau_{k+1}}\big\|^2 \Big|\mathcal{F}_{t_{N^{(\tau_{k+1})}-1}}\Big] \,\,\leq\,\, \big\|\Ejump{t_{N^{(\tau_{k+1})}-1}}\big\|^2  
\\+ \Gamma_1\int_{t_{N^{(\tau_{k+1})}-1}}^{\tau_{k+1}} \mathbb{E}\Big[\big\|   \Ejump{r}\big\|^2\Big|\mathcal{F}_{t_{N^{(\tau_{k+1})}-1}} \Big]dr
\\+\GamRX\hmax^{2\order+1}+(1+2\Czeta{1})\mathbb{E}\left[\left\| \EjumpBefore{\tau_{k+1}}\right\|^2\Big|\mathcal{F}_{t_{N^{(\tau_{k+1})}-1}} \right]. 
\end{multline}
Multiplying by the indicator function on the both sides of \eqref{eq:combine_error_noJ} and \eqref{eq:combine_error_withJ}, and since $\EjumpBefore{t_{n+1}}=\Ejump{t_{n+1}}$ for $n\in[N^{(\tau_k)},N^{(\tau_{k+1})}-2]$, we sum up all the steps from $\tau_k$ to $\tau_{k+1}$ to have 
\begin{eqnarray}
\mathcal{K}_k&:=&\sum_{n=N^{(\tau_k)}}^{N^{(\tau_{k+1})}-1}\bigg(\mathbb{E}\Big[\big\| \Ejump{t_{n+1}}\big\|^2  \Big|\mathcal{F}_{t_n}\Big]-\big\|\Ejump{t_{n}}\big\|^2  \bigg)\mathcal{I}_{\{N^{(\tau_{k+1})}>n\}}\nonumber  \\
&\leq&\Gamma_1\sum_{n=N^{(\tau_k)}}^{N^{(\tau_{k+1})}-1}\int_{t_n}^{t_{n+1}} \mathbb{E}\Big[\big\|   \Ejump{r}\big\|^2\Big|\mathcal{F}_{t_n} \Big]dr\,\mathcal{I}_{\{N^{(\tau_{k+1})}>n\}}\label{eq:LHS_define}
\\&&+\GamRX \sum_{n=N^{(\tau_k)}}^{N^{(\tau_{k+1})}-1}\hmax^{2\order+1}\mathcal{I}_{\{N^{(\tau_{k+1})}>n\}}\nonumber
\\&&+(1+2\Czeta{1})\mathbb{E}\left[\big\|\EjumpBefore{\tau_{k+1}} \big\|^2\middle|\mathcal{F}_{t_{N^{(\tau_{k+1})}-1}}\right]. \nonumber
\end{eqnarray}
For the integrand in the first term on the RHS of \eqref{eq:LHS_define}, by Definition \ref{def:N}, the filtration $\mathcal{F}_{t_n}$ may be written as $\mathcal{F}_{t_{N^{(r)}-1}}$, for $r\in[t_n,t_{n+1})$. By bounding the number of steps between the \nth{k} and \nth{(k+1)} above by $ \pi_{k+1}/\hmin+1$, and bounding all indicator functions above by 1 we have
\begin{multline}\label{eq:oneJump_error}
\mathcal{K}_k\leq\Gamma_1 \int_{\tau_k}^{\tau_{k+1}} \mathbb{E}\bigg[\big\|   \Ejump{r}\big\|^2\bigg|\mathcal{F}_{t_{N^{(r)}-1}} \bigg]dr +\GamRX\left(\rho\pi_{k+1}+1 \right)\hmax^{2\order}\\
    +(1+2\Czeta{1})\mathbb{E}\left[\big\|\EjumpBefore{\tau_{k+1}} \big\|^2\middle|\mathcal{F}_{t_{N^{(\tau_{k+1})}-1}}\right], 
\end{multline}
which is the error bound between the \nth{k} and the \nth{(k+1)} jump, with LHS defined in \eqref{eq:LHS_define}.

\textbf{\textit{Step 3 -- $t_0$ to last jump error bound. }} 
Summing up \eqref{eq:oneJump_error} over $k$ for all jumps that have occured on the interval $(0,t]$, denoted $\NumJumps{t}$ we have 
\begin{multline}
    \sum_{k=0}^{\NumJumps{t}-1}\mathcal{K}_k
    \leq  \Gamma_1 \sum_{k=0}^{\NumJumps{t}-1}\int_{\tau_k}^{\tau_{k+1}} \mathbb{E}\bigg[\big\|   \Ejump{r}\big\|^2\bigg|\mathcal{F}_{t_{N^{(r)}-1}} \bigg]dr\\ +\GamRX\sum_{k=0}^{\NumJumps{t}-1}\left(\rho\pi_{k+1}+1 \right)\hmax^{2\order}\\
    +(1+2\Czeta{1})\sum_{k=0}^{\NumJumps{t}-1}\mathbb{E}\left[\big\|\EjumpBefore{\tau_{k+1}} \big\|^2\middle|\mathcal{F}_{t_{N^{(\tau_{k+1})}-1}}\right]. \label{eq:0_to_J(t)}
\end{multline}
Notice that we set $\tau_0=t_0$ so that \eqref{eq:0_to_J(t)} includes the time from $t_0$ to the first jump $\tau_1$.  Taking expectation on the both sides of \eqref{eq:0_to_J(t)}, conditioned on $\mathcal{H}_T$ (see Section \ref{sec:adpt_strategy_SDE_jump}) that only contains the information of all jump times and jump sizes known by time $T$, we have  
\begin{multline*}
   \mathbb{E}\left[ \sum_{k=0}^{\NumJumps{t}-1}\mathcal{K}_k\middle| \mathcal{H}_T\right] \leq\,\,  \Gamma_1 \mathbb{E}\left[\sum_{k=0}^{\NumJumps{t}-1}\int_{\tau_k}^{\tau_{k+1}} \mathbb{E}\bigg[\big\|   \Ejump{r}\big\|^2\bigg|\mathcal{F}_{t_{N^{(r)}-1}} \bigg]dr\middle| \mathcal{H}_T\right]
\\+\mathbb{E}\left[\GamRX\sum_{k=0}^{\NumJumps{t}-1}\left(\rho\pi_{k+1}+1 \right)\middle| \mathcal{H}_T\right]\hmax^{2\order}
\\+(1+2\Czeta{1})\mathbb{E}\left[\sum_{k=0}^{\NumJumps{t}-1}\mathbb{E}\left[\big\|\EjumpBefore{\tau_{k+1}} \big\|^2\middle|\mathcal{F}_{t_{N^{(\tau_{k+1})}-1}}\right]\middle| \mathcal{H}_T\right].
\end{multline*}
Since $\NumJumps{t}$ is $\mathcal{H}_T$-measurable (see Definition \ref{def:N_bar}), we can take out the summation on the RHS out of the conditional expectation. By $\mathcal{H}_T\subseteq\mathcal{F}_{t_{N^{(r)}-1}}$ for $r\in[0,T]$, with  the tower property and $\tau_{\NumJumps{t}}=\sum_{i=1}^{\NumJumps{t}}\pi_i$  we have
\begin{multline}
   \mathbb{E}\left[ \sum_{k=0}^{\NumJumps{t}-1}\mathcal{K}_k\middle| \mathcal{H}_T\right]
    \leq\,\,  \Gamma_1 \int_{0}^{\tau_{\NumJumps{t}}} \mathbb{E}\bigg[\big\|   \Ejump{r}\big\|^2 \bigg| \mathcal{H}_T\bigg]dr+\Gamma^J_2\left(\rho\tau_{\NumJumps{t}}+1 \right)\hmax^{2\order}\\
    +(1+2\Czeta{1})\sum_{k=0}^{\NumJumps{t}-1}\mathbb{E}\left[\big\|\EjumpBefore{\tau_{k+1}}\big\|^2\middle|\mathcal{H}_T\right], \label{eq:after_HT}
\end{multline}
where $\Gamma^J_2$ is defined in the statement following \eqref{eq:combine_error_noJ}. 

\textbf{\textit{Step 4 --  $\tau_{\NumJumps{t}}$ to $t$ error bound.}} Notice that in the case that the last jump lands at the target time, i.e. $\tau_{\NumJumps{t}}=t$, the whole of Step 4 is not needed. The period reaching $t$ after the last jump at $\tau_{\NumJumps{t}}$ consist of errors only for the diffusion process, so that $\EjumpBefore{t_{n+1}}=\Ejump{t_{n+1}}$ for $n\in[N^{(\tau_{\NumJumps{t}})}, t_{N^{(t)}-1}]$. Since $t\in\big[t_{N^{(t)}-1},t_{N^{(t)}}\big]$, by replacing $t_n,t_{n+1}$ with $t_{N^{(t)}-1}$ and $t$ respectively in \eqref{eq:combine_error_noJ}, we have the last step reaching $t$ as
\begin{multline}
    \mathbb{E}\Big[\big\|\Ejump{t}\big\|^2 \Big|\mathcal{F}_{t_{N^{(t)}-1}}\Big] \leq \big\|\Ejump{t_{N^{(t)}-1}}\big\|^2\\
    +\Gamma_1\int_{t_{N^{(t)}-1}}^{t} \mathbb{E}\Big[\big\|\Ejump{r}\big\|^2 \Big|\mathcal{F}_{t_{N^{(t)}-1}}\Big]dr
+\GamRX {\big|t-t_{N^{(t)}-1}\big|^{2\order+1}}. \label{eq:err_extra}
\end{multline}
Multiplying the indicator function on the both sides of \eqref{eq:combine_error_noJ}, summing up to $t$ with the last step \eqref{eq:err_extra} and taking the expectation conditioned on $\mathcal{H}_T$ we have
\begin{align*}\label{eq:LHSlast_define}
\mathcal{R}(t)\,\,:=\,\, & \mathbb{E}\Bigg[\,\sum_{n=N^{(\tau_{\NumJumps{t}})}}^{N^{(t)}-2}\Big(\mathbb{E}\Big[\big\|\Ejump{t_{n+1}}\big\|^2 \Big|\mathcal{F}_{t_n}\Big]-\big\|\Ejump{t_{n}}\big\|^2\Big)\mathbf{1}_{\{N^{(t)}> n+1\}} \numberthis
\\
&\qquad\qquad\quad+\mathbb{E}\Big[\big\|\Ejump{t}\big\|^2 \Big|\mathcal{F}_{t_{N^{(t)}-1}}\Big]-\big\|\Ejump{t_{N^{(t)}-1}}\big\|^2 \Bigg|\mathcal{H}_T \Bigg] 
\\ 
\leq\,\,&\Gamma_1\,\mathbb{E}\Bigg[\,  \sum_{n=N^{(\tau_{\NumJumps{t}})}}^{N^{(t)}-2}\int_{t_n}^{t_{n+1}}\mathbb{E}\Big[\big\| \Ejump{r}\big\|^2 \Big|\mathcal{F}_{t_n}\Big]\mathbf{1}_{\{N^{(t)}> n+1\}}dr 
\\
&\qquad\qquad\qquad\qquad+\int_{t_{N^{(t)}-1}}^{t} \mathbb{E}\Big[\big\| \Ejump{r}\big\|^2 \Big|\mathcal{F}_{t_{N^{(t)}-1}}\Big]dr\Bigg|\mathcal{H}_T\Bigg] 
\\
&+\mathbb{E}\Bigg[\GamRX \Bigg(\sum_{n=N^{(\tau_{\NumJumps{t}})}}^{N^{(t)}-2} {h_{n+1}^{2\order+1}}\mathbf{1}_{\{N^{(t)}> n+1\}}+{\big|t-t_{N^{(t)}-1}\big|^{2\order+1}}\Bigg)\Bigg|\mathcal{H}_T\Bigg]
\end{align*}
By summing the integrals on the RHS of \eqref{eq:LHSlast_define}, applying the tower property of conditional expectations, and bounding indicator functions above by 1 we get 
\begin{align*}
\mathcal{R}(t)\leq \Gamma_1 \displaystyle \int_{\tau_{\NumJumps{t}}}^{t}\mathbb{E}\Big[\big\| \Ejump{r}\big\|^2 \Big|\mathcal{H}_{T}\Big]dr+\Gamma^J_2\big(\rho (t-\tau_{\NumJumps{t}}) +1\big) {\hmax^{2\order}}. \numberthis \label{eq:last_period}
\end{align*}

\textbf{\textit{Step 5 -- $t_0$ to $t$ error bound.}} Adding the error bound after the last jump $\mathcal{R}(t)$ (defined in \eqref{eq:LHSlast_define} and bounded in \eqref{eq:last_period}) to the error bound at the last jump \eqref{eq:after_HT} we have 
\begin{multline}
   \mathbb{E}\left[ \sum_{k=0}^{\NumJumps{t}-1}\mathcal{K}_k\middle| \mathcal{H}_T\right]+ \mathcal{R}(t)
    \leq\,\,  \Gamma_1\int_{0}^{t} \mathbb{E}\bigg[\big\|   \Ejump{r}\big\|^2 \bigg| \mathcal{H}_T\bigg]dr\\
    +\Gamma^J_2\left(\rho t+1 \right){\hmax^{2\order}}+(1+2\Czeta{1})\sum_{k=0}^{\NumJumps{t}-1}\mathbb{E}\left[\big\|\EjumpBefore{\tau_{k+1}} \big\|^2\middle|\mathcal{H}_T\right], \label{eq:after_HT_withLast}
\end{multline}
where $\mathcal{K}_k$ is defined in \eqref{eq:LHS_define}, respectively. Then for the left hand side of \eqref{eq:after_HT_withLast}, we first combine the two sums on $k$ and $n$ to one sum on $n$. 
\begin{align*}
    &\mathbb{E}\left[ \sum_{k=0}^{\NumJumps{t}-1}\mathcal{K}_k\middle| \mathcal{H}_T\right]+ \mathcal{R}(t)\\
    =\,\,&\mathbb{E}\Bigg[ \sum_{k=0}^{\NumJumps{t}-1}\Bigg(\sum_{n=N^{(\tau_k)}}^{N^{(\tau_{k+1})}-1}\bigg(\mathbb{E}\Big[\big\| \Ejump{t_{n+1}}\big\|^2  \Big|\mathcal{F}_{t_n}\Big]-\big\|\Ejump{t_{n}}\big\|^2  \bigg)\mathcal{I}_{N^{(\tau_{k+1})}>n} \\
    &\qquad+ \displaystyle \sum_{n=N^{(\tau_{\NumJumps{t}})}}^{N^{(t)}-2}\Big(\mathbb{E}\Big[\big\|\Ejump{t_{n+1}}\big\|^2 \Big|\mathcal{F}_{t_n}\Big]-\big\|\Ejump{t_{n}}\big\|^2\Big)\mathbf{1}_{\{N^{(t)}> n+1\}} \\
    &\,\,\quad\quad\qquad\quad\qquad\qquad+\mathbb{E}\Big[\big\|\Ejump{t}\big\|^2 \Big|\mathcal{F}_{t_{N^{(t)}-1}}\Big]-\big\|\Ejump{t_{N^{(t)}-1}}\big\|^2 \Bigg|\mathcal{H}_T \Bigg]\\
    =\,\,&\mathbb{E}\Bigg[ \sum_{n=0}^{N^{(t)}-2}\bigg(\mathbb{E}\Big[\big\| \Ejump{t_{n+1}} )\big\|^2  \Big|\mathcal{F}_{t_n}\Big]-\big\|\Ejump{t_{n}}\big\|^2  \bigg)\mathcal{I}_{N^{(t)}>n+1}\\
    &\,\,\quad\quad\qquad\quad\qquad\qquad+\mathbb{E}\Big[\big\|\Ejump{t }\big\|^2 \Big|\mathcal{F}_{t_{N^{(t)}-1}}\Big]-\big\|\Ejump{t_{N^{(t)}-1}}\big\|^2 \Bigg|\mathcal{H}_T \Bigg].
\end{align*}
Since $ N_{\max}^{(t)}$  is $\mathcal{H}_T$-measurable, we bound $N^{(t)}$ by $ N_{\max}^{(t)}$ and move the sum out of the conditional expectation. Then, by Section \ref{sec:adpt_strategy_SDE_jump} that $\mathcal{H}_T$ is a sub-$\sigma$-algebra of $\mathcal{F}_{t_n}$ for $n\in\big[0,N^{(t)}_{\max}-1\big ]$, we apply tower property. Finally, from the telescoping sum  with $\Ejump{0}=0$ we have 
\begin{align*}
    &\mathbb{E}\left[ \sum_{k=0}^{\NumJumps{t}-1}\mathcal{K}_k\middle| \mathcal{H}_T\right]+ \mathcal{R}(t)\\
    =& \sum_{n=0}^{N^{(t)}_{\max}-2}\mathbb{E}\Bigg[\mathbb{E}\Big[\big\| \Ejump{t_{n+1}} \big\|^2 \mathcal{I}_{N^{(t)}>n+1} \Big|\mathcal{F}_{t_n}\Big]-\big\|\Ejump{t_{n}} \big\|^2\mathcal{I}_{N^{(t)}>n+1}  \\
    &\,\,\quad\quad\qquad\quad\qquad\qquad+\mathbb{E}\Big[\big\|\Ejump{t } \big\|^2 \Big|\mathcal{F}_{t_{N^{(t)}-1}}\Big]-\big\|\Ejump{t_{N^{(t)}-1}}\big\|^2 \Bigg|\mathcal{H}_T \Bigg]\\
    =\,\,&\mathbb{E}\Big[\big\|\Ejump{t_{N^{(t)}-1}}\big\|^2 \Big|\mathcal{H}_T \Big]-\mathbb{E}\Big[\big\|\Ejump{t_0}\big\|^2 \Big|\mathcal{H}_T\Big]\\
    &\,\,\quad\quad\qquad\quad\qquad\qquad+\mathbb{E}\Big[\big\|\Ejump{t }\big\|^2 \Big|\mathcal{H}_T\Big]-\mathbb{E}\Big[\big\|\Ejump{t_{N^{(t)}-1}}\big\|^2 \Big|\mathcal{H}_T \Big]\\
    =\,\,&\mathbb{E}\Big[\big\|\Ejump{t }\big\|^2\Big| \mathcal{H}_T \Big]. \numberthis \label{eq:LHS_total}
\end{align*}
where $N^{(t)}_{\max}:= \lceil t/\hmin + \NumJumps{t} \rceil$. Substituting \eqref{eq:LHS_total} back to \eqref{eq:after_HT_withLast}, we have 
\begin{multline}
    \mathbb{E}\Big[\big\|\Ejump{t}\big\|^2\Big| \mathcal{H}_T \Big]\leq \Gamma_1 \int_{0}^{t} \mathbb{E}\bigg[\big\|   \Ejump{r}\big\|^2 \bigg| \mathcal{H}_T\bigg]dr+\GamR\left(\rho t+1 \right){\hmax^{2\order}}\\
    +(1+2\Czeta{1})\sum_{k=0}^{\NumJumps{t}-1}\mathbb{E}\left[\big\|\EjumpBefore{\tau_{k+1}} \big\|^2\middle|\mathcal{H}_T\right]. \label{eq:final_HT}
\end{multline}
By \eqref{eq:jump_sum=int}, we write the last jump term in the integral form, take a final expectation, and use the fact that $\nu(Z)=\lambda$ to get
\begin{align*}
    \mathbb{E}\left[\sum_{k=0}^{\NumJumps{t}-1}\mathbb{E}\left[\big\|\EjumpBefore{\tau_{k+1}}  \big\|^2\middle|\mathcal{H}_T\right]\right]=\,\,&\mathbb{E}\left[\int_{0}^{t}\int_Z\mathbb{E}\left[\big\|\EjumpBefore{r}  \big\|^2 \middle|\mathcal{H}_T\right]J_{\nu}(dz\times dr)\right]\\
    \leq\,\,&\int_{0}^{t}\int_Z\mathbb{E}\left[\big\|\EjumpBefore{r} \big\|^2 \right] \nu(dz)dr\\
    \leq\,\,&\lambda \int_{0}^{t} \mathbb{E}\left[\big\|\Ejump{r} \big\|^2 \right]  dr. \numberthis \label{eq:final_jump}
\end{align*}
Substituting \eqref{eq:final_jump} back into \eqref{eq:final_HT} and take a final expectation on the both sides we have 
\begin{multline*}
    \mathbb{E}\bigg[\mathbb{E}\Big[\big\|\Ejump{t}\big\|^2\Big| \mathcal{H}_T \Big]\bigg]\leq \Gamma_1  \mathbb{E}\left[\int_{0}^{t} \mathbb{E}\Big[\big\|    \Ejump{r}\big\|^2 \Big| \mathcal{H}_T\Big]dr\right]+\GamR\left(\rho t+1 \right){\hmax^{2\order}}\\
    +  \lambda (1+2\Czeta{1})\int_{0}^{t} \mathbb{E}\left[\big\|\Ejump{r} \big\|^2 \right]  dr.
\end{multline*}
Simplifying with tower property we have 
\begin{align*}
    \mathbb{E}\Big[\big\|\Ejump{t}\big\|^2 \Big]
    \leq \Big(\Gamma_1 + \lambda (1+2\Czeta{1})\Big)  \int_{0}^{t} \mathbb{E}\Big[\big\|   \Ejump{r}\big\|^2  \Big]dr +\GamR\left(\rho t+1 \right){\hmax^{2\order}}.
\end{align*}
For all $t\in[0,T]$, by Gr\"onwall's inequality we have 
\begin{align*}
    \Big(\mathbb{E}\Big[\big\|\Ejump{t}\big\|^2 \Big]\Big)^{1/2}\leq \finalC{t}\,{\hmax^\order}. 
\end{align*}
Taking the maximum over $t$ on the both sides, the proof follows with
\begin{align}   \finalC{T}:=\sqrt{\GamR\left(\rho T+1 \right)\exp\Big(T\big(\Gamma_1 + \lambda (1+2\Czeta{1})\big)\Big)}. \label{eq:CJRrhot}
\end{align}
\end{proof}

\subsection{An example of a main map: the Jump Adapted Adaptive Milstein Method (\JAAMM)}\label{sec:Mil}
Now recall as an example the Milstein method expressed as a map, \cite{mao2007SDEapp,Fandithesis} for the SDE \eqref{eq:true}. Over each step $[t_n,t_{n+1}]$ the Milstein map $\mainMap:\mathbb{R}^d\times \mathbb{R}  \times\mathbb{R}\rightarrow \mathbb{R}^d$ is defined as
\begin{align}
\mainMap\big(x,t_n,s-t_n\big):= x+(s-t_n)f(x)+\sum_{i=1}^{m}g_i(x)I_{i}^{t_n,s}+\sum_{i,j=1}^{m}\mathbf{D}g_i(x)g_j(x) I_{j,i}^{t_n,s}. \numberthis \label{eq:deftheta}
\end{align}
where the iterated stochastic integrals are defined as 
\begin{align}
    I_{i}^{t_n,t_{n+1}}:=\int_{t_{n}}^{t_{n+1}}dW_i(s), \qquad I_{j,i}^{t_n,t_{n+1}}:=\int_{t_{n}}^{t_{n+1}} \int_{t_{n}}^{s}dW_j(p) dW_i(s). \label{eq:defISI}
\end{align}

\begin{assumption} \label{ass:SDEmoments_power}
Suppose that, for some constants $c_{  3,4,5,6  }$, $q_1$, $q_2\geq 0$; $i=1,\dots,m$, we have
\begin{alignat}{2}
   \big \|\mathbf{D}f(x)\big\|_{\mathbf{F}}\leq\,\,& c_3(1+\|x\|^{q_1+1}), \qquad\quad \big\|\mathbf{D}g_i(x)\big\|_{\mathbf{F}}&&\leq\,\, c_4(1+\|x\|^{q_2+1}), \label{eq:Df+Dg}\\
    \big\| f(x)\big\| \leq\,\,& c_5(1+\|x\|^{q_1+2}), \quad\quad \big\| g(x)\big\|_{\mathbf{F}(d\times m)} &&\leq\,\, c_6(1+\|x\|^{q_2+2}).\label{eq:||f||+||g||}
\end{alignat}
Furthermore, for some $c_{1,2}\geq 0$; $i=1,\dots,m$, we have
\begin{align}
    \big\|\mathbf{D}^2f(x)\big\|_{\mathbf{T}_3}\leq\,\,
    c_1(1+\|x\|^{q_1}), \quad \big\|\mathbf{D}^2g_i(x)\big\|_{\mathbf{T}_3}\leq\,\, c_2(1+\|x\|^{q_2}).\label{eq:D^2f+D^2g}
\end{align}
\end{assumption}

\begin{assumption}\label{ass:etaq}
Suppose that \eqref{eq:monotone} in Assumption \ref{ass:f+g} holds with 
\begin{align*}
    \eta \geq 4q  + 2q_2+10,
\end{align*}
where $q:=q_1\vee q_2$, $q_1$ and $q_2$ are from \eqref{eq:||f||+||g||} in Assumption \ref{ass:SDEmoments_power}.
\end{assumption}

The following moment bounds then apply over any finite interval $[0,T]$:
\begin{lemma}\label{lem:boundedMomentsSDE_jump}
Let $f$, $g$ and $\gamma$ satisfy  \eqref{eq:local_lipschitz},\eqref{eq:monotone}, \eqref{eq:local_lipschitz_jump}, \eqref{eq:monotone_gamma} and \eqref{eq:||f||+||g||}. Then the SJDE \eqref{eq:true_jump} has a unique global solution such that there exists a constant 
\begin{equation}\label{eq:SDEmoments_jump}
\mathbb{E}\biggl[\sup_{s\in[0,T]}\|\Xjump{s}\|^{\eta-2q_2-2}\biggr] \leq \CxJ,
\end{equation}
with $\CxJ:=\CxJ(t,X_0,\eta,q_2,\lambda)$.
\end{lemma}
\begin{proof}
The proof, inspired by \cite[Lem. 4.2]{mao2015truncated} and \cite[Lem. 2.2]{dareiotis2016tamed}, makes use of \cite[Lem. 3.5]{deng2019truncated} and may be found in full in \cite{Fandithesis}. 
\end{proof}
    
\begin{remark}\label{rem:SDE_power_jump}
The moment in \eqref{eq:SDEmoments_jump} only depends on $\eta$ and $q_2$ because the diffusion coefficient $g$ in \eqref{eq:true_jump} is superlinearly bounded whereas the the jump coefficient $\gamma$ in \eqref{eq:monotone_gamma} is not, see Assumptions \ref{ass:f+g} and \ref{ass:jump_gamma}. 
\end{remark}

To ensure the appropriate bound to satisfy Assumption \ref{ass:one-stepBounds}, we use a particular kind of adaptive time-stepping strategy, such that whenever $\hmin < h_{n+1} \leq \hmax$,
 \begin{align}
  \big\|\Yjump{t_n}\big\| < R, \quad n=0,\dots, N-1. \label{eq:defYn_jump}
 \end{align}
This is similar to the path-bounded strategy for adaptive Milstein method in \cite{KLSBit,Fandithesis}.

\begin{lemma}\label{lem:stratPB_jump}
Let $\big\{\big(\Yjump{s}\big)_{s\in[t_n,t_{n+1}]},h_{n+1}\big\}_{n\in\mathbb{N}}$ be the \JAAMM scheme 
with map $\mathcal{M}$ from Definition \ref{def:j-adapted adaptive explicit Milstein scheme} given by \eqref{eq:deftheta}. 
Then $\{h_{n+1}\}_{n\in \mathbb{N}}$ satisfies \eqref{eq:defYn_jump} if for each $n = 0, \dots , N-1$ and $\kappa>0$,  
\begin{align}
    h_{n+1}=\left(\hmin \vee \left ( \frac{\hmax}{\big\|\Yjump{t_n}\big\|^{1/\kappa}} \wedge \hmax \right ) \right)\wedge \Big(\tau_{\NumJumps{t_n}+1}-t_n\Big) . \label{eq:defh_jump}
\end{align}
\end{lemma}
\begin{proof}

When we compute $h_{n+1}$ according to \eqref{eq:defh_jump} there are four possible cases, leading to the following partition of the event $\{\hmin < h_{n+1} \leq \hmax\}$.
\begin{multline*}
\left\{\hmin < h_{n+1} \leq \hmax\right\}\\
=\underbrace{\left\{h_{n+1}=\frac{\hmax}{\big\|\Yjump{t_n}\big\|^{1/\kappa}}\right\}}_{=:A}
\cup\underbrace{\left\{h_{n+1}=\tau_{\NumJumps{t_n}+1}-t_n\right\}}_{=:B}\cup\underbrace{\left\{h_{n+1}=h_{\max}\right\}}_{=:C}\cup\underbrace{\left\{h_{n+1}=h_{\min}\right\}}_{=:D}
\end{multline*}
Event $D$ corresponds to use of the backstop method so it may be disregarded. On event $B$ it follows from \eqref{eq:defh_jump} that
\[
h_{n+1}=\tau_{\NumJumps{t_n}+1}-t_n<\frac{\hmax}{\big\|\Yjump{t_n}\big\|^{1/\kappa}}.
\] 
On event $C$, it follows from \eqref{eq:defh_jump} that
\[
h_{n+1}=h_{\max}<\frac{\hmax}{\big\|\Yjump{t_n}\big\|^{1/\kappa}}.
\]
We can therefore say that on the event $A\cup B\cup C$ the relation
\[
h_{\min}<h_{n+1}<\frac{\hmax}{\big\|\Yjump{t_n}\big\|^{1/\kappa}}
\]
holds. Rearranging, and with $\rho\hmin=\hmax$ we have 
\begin{equation*}
    \big\|\Yjump{t_n}\big\| <  \left(\frac{\hmax}{\hmin}\right)^\kappa <  \left(\frac{\hmax}{\hmin}\right)^\kappa = \rho^\kappa,
\end{equation*}
so \eqref{eq:defYn_jump} is satisfied with  $R=\rho^\kappa$.
\end{proof}

We can then prove the following lemma, which confirms that the Milstein method satisfies \eqref{eq:Case_I_twostep} in Assumption \ref{ass:one-stepBounds} with $\order=1$.
\begin{lemma}  \label{lmm:CaseI_jump}
Let $f$, $g$ satisfy Assumption \ref{ass:f+g}, \ref{ass:SDEmoments_power}, and \ref{ass:etaq}. Let $\Xjump{s}$ be a solution of \eqref{eq:true_jump} with $\gamma\equiv 0$. Suppose $\EThetaJump{s}$ defined in \eqref{eq:defE(r)_jump} for $s\in[t_n,t_{n+1}]$, $n\in\mathbb{N}$. 
There exists a constant $C_E$ and integrable $\mathcal{F}_{t_n}$-measurable random variables $C_{M,n}$ such that for $A_n,B_n$ a.s. finite and $\mathbb{R}^d$-valued random variables with $\|B_n\|\leq R$ a.s. and $R$ is as given in \eqref{eq:defYn_jump}:
\begin{multline}
    \mathbb{E}\Big[\Big\|  \EThetaJump{t_{n+1}})\Big\|^2 \Big|\mathcal{F}_{t_n}\Big] \leq \|A_n-B_n\|^2\\ 
    +C_E\int_{t_n}^{t_{n+1}} \mathbb{E}\left[\big\|\EThetaJump{r}\big\|^2 \middle|\mathcal{F}_{t_n}\right]dr
    +C_{M,n}\,h_{n+1}^3, \quad a.s.  \label{eq:Case_I_twostep2}
\end{multline} 
\end{lemma}
\begin{proof}
The proof is the same of Lemma 7.2 in ~\cite{KLSBit}, where $X(t_n)$ and $\widetilde Y(t_n)$ have been replaced with $x$ and $y$ using the fact that we are applying the map $\mainMap$ over an interval where no jump takes place. In order to ensure finiteness of the expectation of $C_{M,n}$ certain moments of the SDE \eqref{eq:true} must be finite, and this is ensured by Assumptions \ref{ass:SDEmoments_power} and \ref{ass:etaq}:  also see~\cite{KLSBit}.
\end{proof}
Finally, we comment that although the bound \eqref{eq:Case_I_twostep2} only holds for $\|B_n\|\leq R$, this is sufficient to ensure the conditions of Assumption \ref{ass:one-stepBounds}, since if $\|B_n\|>R$ then $h_{n+1}\leq h_{\min}$ and $\EThetaJump{t_{n+1}}=0$. For the backstop map $\mathcal{B}$, we require a scheme that satisfies a mean-square consistency condition of the same order leading to \eqref{eq:Case_I_twostep2}. In Section \ref{sec:numerics} we use order-one strongly $L_2$-convergent method as backstop, namely the projected Milstein method (see \cite{beyn2017stochastic}). Together these maps define our \JAAMM for numerical simulations below. 

In fact we can show that the probability of actually using the backstop map $\mathcal{B}$ can be made arbitrarily small by controlling the parameter $\rho$.
\begin{theorem}\label{thm:backstopProb}
Suppose that all the conditions of Theorem \ref{thm:result_jump} hold, and suppose that the time-stepping strategy satisfies \eqref{eq:defh_jump} in the statement of Lemma \ref{lem:stratPB_jump}, so that \eqref{eq:defYn_jump} holds for some $R<\infty$. For any fixed $\kappa\geq 1$ there exists a constant $C_{\text{prob}}=C_{\text{prob}}(T,R,h_{\max})$ such that for $h_{\max}\leq 1/\finalC{T}$, where $\finalC{T}$ is as in the statement of Theorem \ref{thm:result_jump},
\begin{equation}
\mathbb{P}[h_{n+1}\leq h_{\min}]\leq C_{\text{prob}}\rho^{1-2\kappa}+(1-e^{-\lambda h_{\max}/\rho}).    
\end{equation}
Further for arbitrarily small tolerance $\varepsilon\in(0,1)$, there exists $\rho>0$ such that
\begin{equation}
\mathbb{P}[h_{n+1}\leq h_{\min}]<\varepsilon,\quad n\in\mathbb{N}.    
\end{equation}
\end{theorem}
\begin{proof}
Decompose the event
\begin{align*}
\left\{h_{n+1}\leq h_{\min}\right\}=\underbrace{\left[\left\{\frac{h_{\max}}{\|Y(t_n)\|}\leq h_{\min}\right\}\cap\left\{\tau_{\bar N_{t_n}+1}-t_n>h_{\min}\right\}\right]}_{=:E}
\cup\underbrace{\left\{\tau_{\bar N_{t_n}+1}-t_n<h_{\min}\right\}}_{=:F}
\end{align*}
Since events $E$ and $F$ are disjoint, we can write
\[
\mathbb{P}[h_{n+1}\leq h_{\min}]=\mathbb{P}[E]+\mathbb{P}[F].
\]
First, event $E$ corresponds to a set of trajectories where no jump has occurred over the step $h_{n+1}$, and therefore the same argument used in the proof of Theorem 4.2 in \cite{KLSBit} along with with the condition \eqref{eq:Case_I_twostep} gives
\[
\mathbb{P}\left[E\right]\leq C_{\text{prob}}\rho^{(1-2\kappa)}
\]
Second, event $F$ may be expressed in terms of the distribution function of an exponential random variable with rate $\lambda$, giving (since $\rho=h_{\max}/h_{\min}$)
\[
\mathbb{P}[F]=1-e^{-\lambda h_{\max}/\rho}.
\]
The second part of the statement of the Theorem follows again from Theorem 4.2 in \cite{KLSBit} and the fact that $\lim_{\rho\to\infty}\mathbb{P}[F]=0$.
\end{proof}

\section{Numerical experiments}\label{sec:numerics}
In this section we confirm by experiment the order-one strong convergence of the \JAAMM method and compare its computational efficiency against several jump adapted fixed-step methods in two ways. First, we use two scalar test examples with (respectively) additive and multiplicative diffusion terms to demonstrate the effect of increasing the jump rate parameter $\lambda$ on the error constant. Second, we use three 2-dimensional test systems with diagonal, commutative, and non-commutative diffusion coefficients respectively to observe the effects of these perturbation structures on computational efficiency in the presence of jumps. 

\subsection{The Monte Carlo estimator and reference solution}
Since our test systems do not have solutions that can be represented explicitly in terms of the diffusion and jump noise processes, we estimate the strong error in $L_2$ using as a reference solution the jump adapted fixed-step projected Milstein method \cite{beyn2017stochastic} over a fine mesh, which is known to be strongly convergent of order one uder the conditions of Assumptions \ref{ass:f+g} and \ref{ass:jump_gamma}.

We approximate the final-time error in mean-square for all methods tested  using a Monte Carlo approach with a reference solution computed over a fine mesh.  
%
To ensure a fair comparison between methods, we first run the adaptive method for all realizations and define 
$$h_{\text{mean}}:=\frac{1}{M}\sum_{m=1}^{M}\frac{T}{N_{m}-\bar{N}_m},
$$
with $N_{m}$ denoting the number of adaptive steps taken on the $m^{th}$ sample path to reach $T$, and $\bar{N}_m$ the total number of jumps on the same path. $h_{\text{mean}}$ is used as the fixed stepsize for all other fixed-step methods that we want to compare. We have subtracted $\bar{N}_m$ from $N_m$ when computing $h_{\text{mean}}$ because jump times will be superimposed on the mesh once again when running fixed-step methods. On our convergence plots we show a reference line of slope one.

To compare computational efficiency, we re-run the adaptive method with other fixed-step methods separately, with independent Wiener increments and without a reference solution. The CPU time consumed for each method on each realisation is recorded as the measurement of the computational time they require. Finally, we take the sample mean of all the CPU times to approximate the expected efficiency.

We compare \JAAMM to the following jump adapted methods : projected Milstein \cite{beyn2017stochastic} with scale parameter $0.25$, denoted \texttt{JA-PMil}, a jump adapted variant of the split-step backward Milstein (\texttt{JA-SSBM}) \cite{beyn2017stochastic} and the tamed Milstein method (\texttt{JA-TMil}) \cite{KumarSabanis2019}.

\subsection{One-dimensional test systems}
In order to demonstrate strong convergence of order one for a scalar test equation with non-globally Lipschitz drift, we use as our test equation
\begin{multline}\label{eq:1D_model2}
\Xjump{t}=\XjumpStart+\int_{0}^{t}\big[\Xjump{s}-3\Xjump{s}^3\big]ds+\int_{0}^{t}g(\Xjump{s})dW(s)\\
+\int_{0}^{t}\int_{Z}\gamma\big(z,\Xjump{r^-}\big)J_\nu(dz\,\times\,dr),\quad t\in[0,1],
\end{multline}
We consider both additive and multiplicative noise.
For the additive noise case we set $g(x)=\sigma$, and for the multiplicative noise case we set $g(x)=\sigma(1-x^2)$ with $\sigma=0.2$. In both cases $\XjumpStart=0.5$, $T=1$, and the jump coefficient is $\gamma(z,x)=zx$ where $z\sim\mathcal{N}(0,0.01)$.%

For the reference solution, we use \texttt{JA-PMil} on a jump adapted fixed-step mesh with reference stepsize $2^{-18}$. In the case where an approximation step lands between two reference steps, we use a Brownian bridge to interpolate the reference Brownian trajectory. Throughout, $\rho=2^7$ (see Assumption \ref{ass:h}).

For both additive (Figure \ref{fig:1DAddMult} (a)-(b)) and multiplicative (Figure \ref{fig:1DAddMult} (c)-(d)) noises, we compare mean-square error for $\hmax=[2^{-14},2^{-13},2^{-12},2^{-11},2^{-10}]$ using $M=500$ sample trajectories. We can see in Figure \ref{fig:1DAddMult} that in both cases, \JAAMM shows advantages in error and CPU time when jumps occur with intensity $\lambda=2$. 

\begin{figure}
    \centering
  (a) \hspace{0.48\textwidth} (b)\\
    \includegraphics[width=0.48\textwidth]{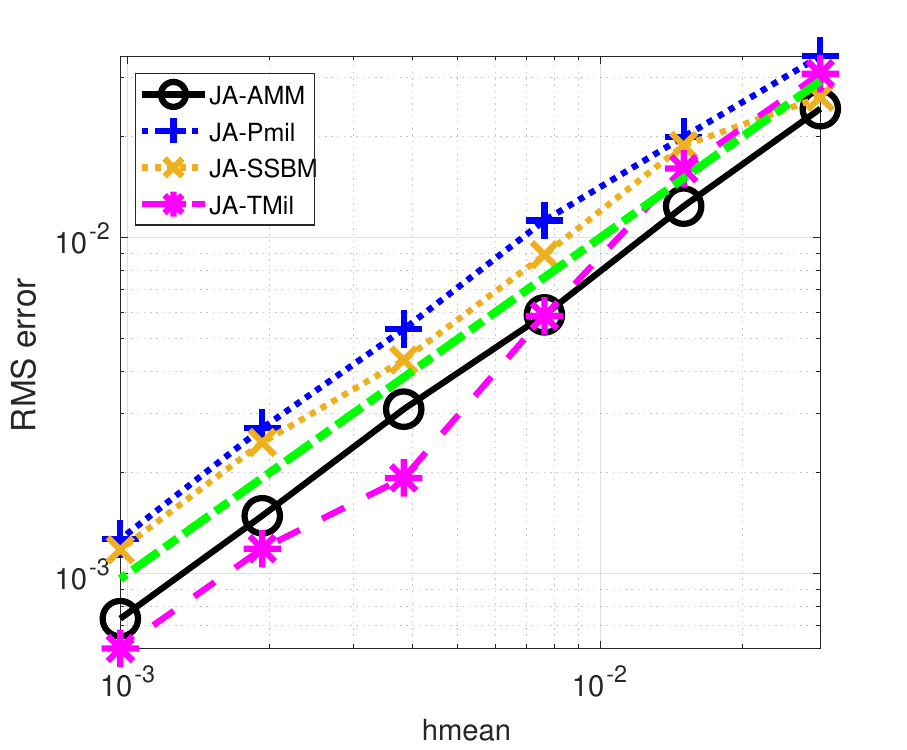}
    \includegraphics[width=0.48\textwidth]{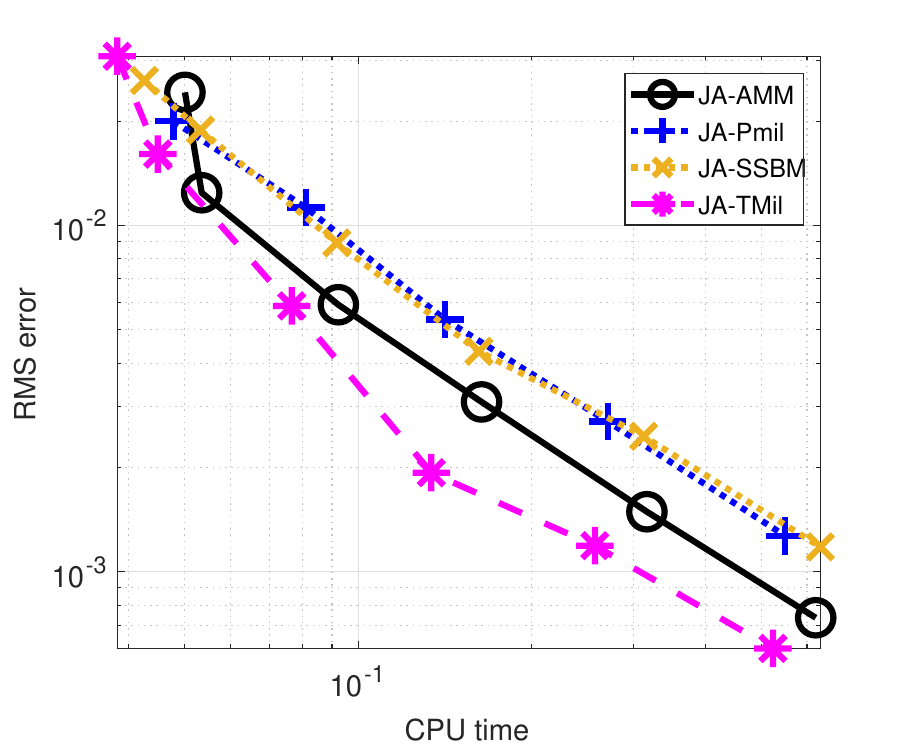}
    
  (c) \hspace{0.48\textwidth} (d)\\
\includegraphics[width=0.48\textwidth]{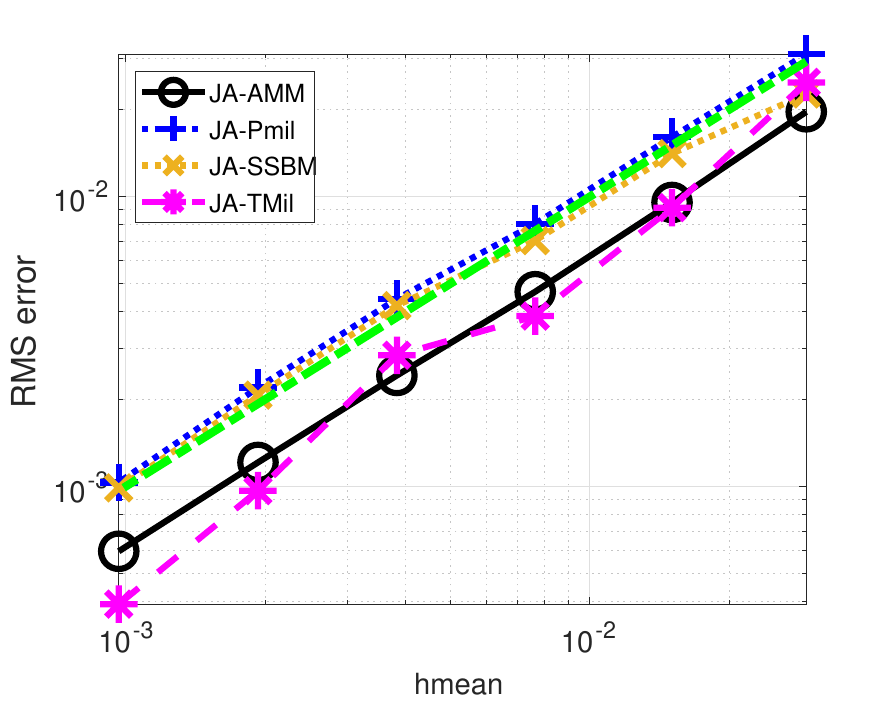}
    \includegraphics[width=0.48\textwidth]{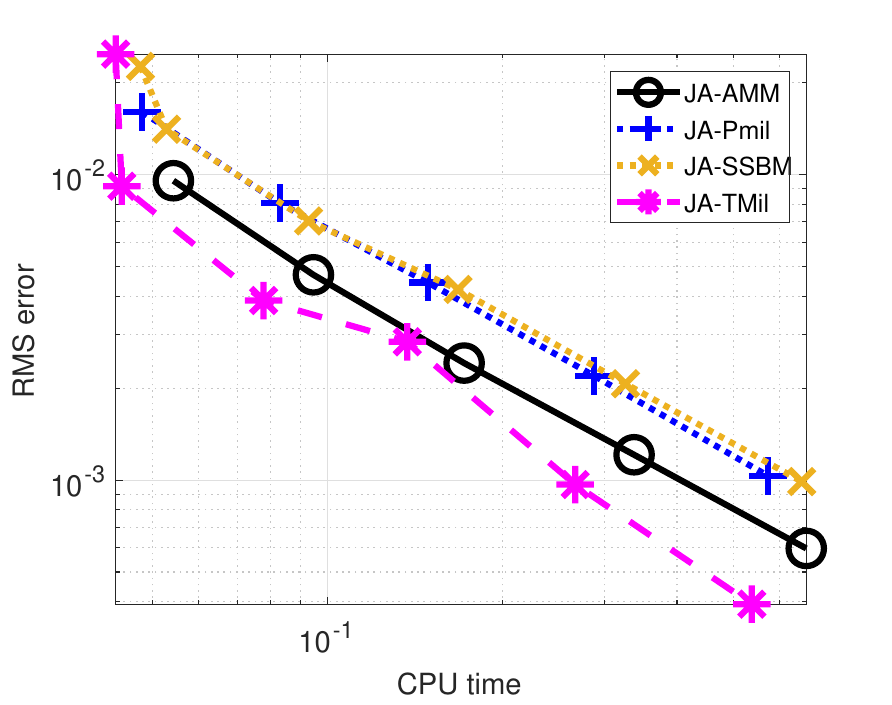}

    \caption{Strong convergence and efficiency of \JAAMM for approximating the one-dimensional system with (a) and (b) for additive noise; (c) and (d) for multiplicative noise. 
    }\label{fig:1DAddMult}
\end{figure}
Further, we demonstrate in Figure \ref{fig:1DLambda} the performance of \JAAMM when the jump intensity $\lambda$ increases. With the same settings for 1D multiplicative model (where $\lambda=2$), we see that \JAAMM improves its relative performance as the jump intensity increases from $25$ to $250$. 
Since \texttt{JA-TMil} displays significantly larger error constants in the case $\lambda=250$ (by at least two orders) we do not include the results in Figures \ref{fig:1DLambda} (c) and (d). For large $\lambda$, for the range of time steps considered, we do not observe convergence either for \texttt{JA-PMil}. \JAAMM shows both the best rate and error constant for $\lambda=250$.

\begin{figure}
    \centering
  (a) \hspace{0.48\textwidth} (b)\\
\includegraphics[width=0.49\textwidth]{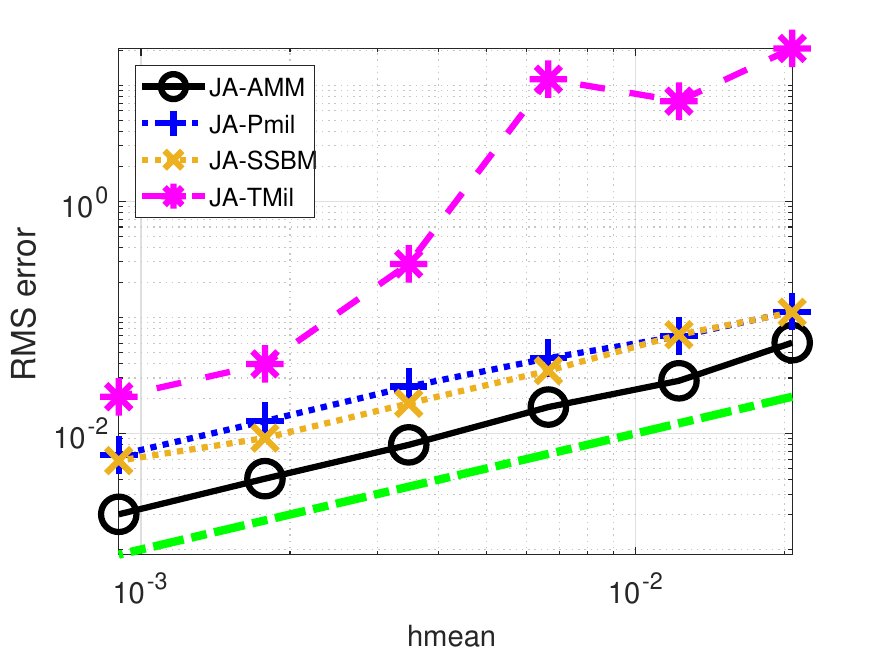}
    \includegraphics[width=0.47\textwidth]{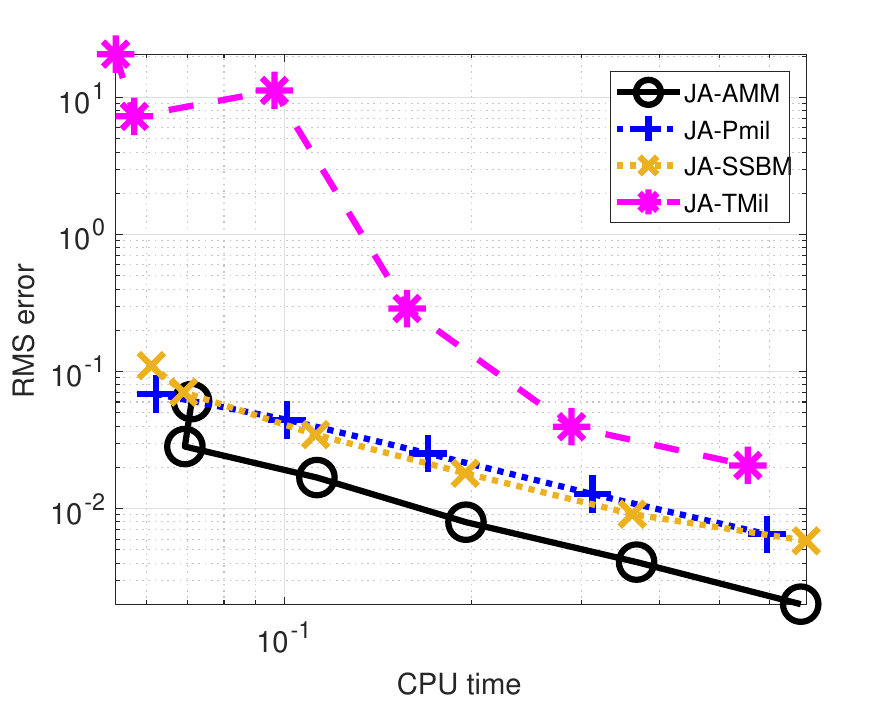}
 (c) \hspace{0.48\textwidth} (d)\\
 \includegraphics[width=0.49\textwidth]{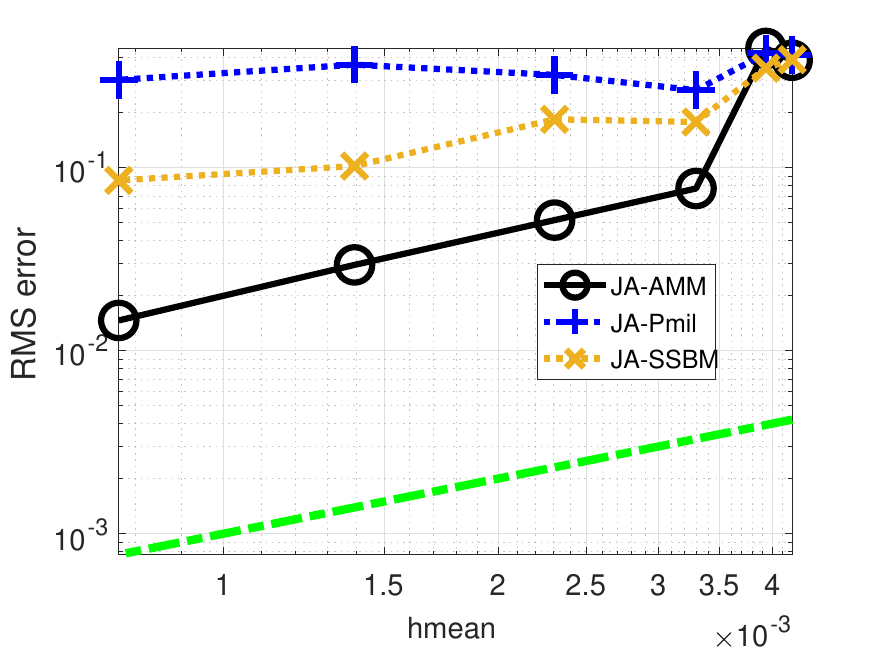}
    \includegraphics[width=0.47\textwidth]{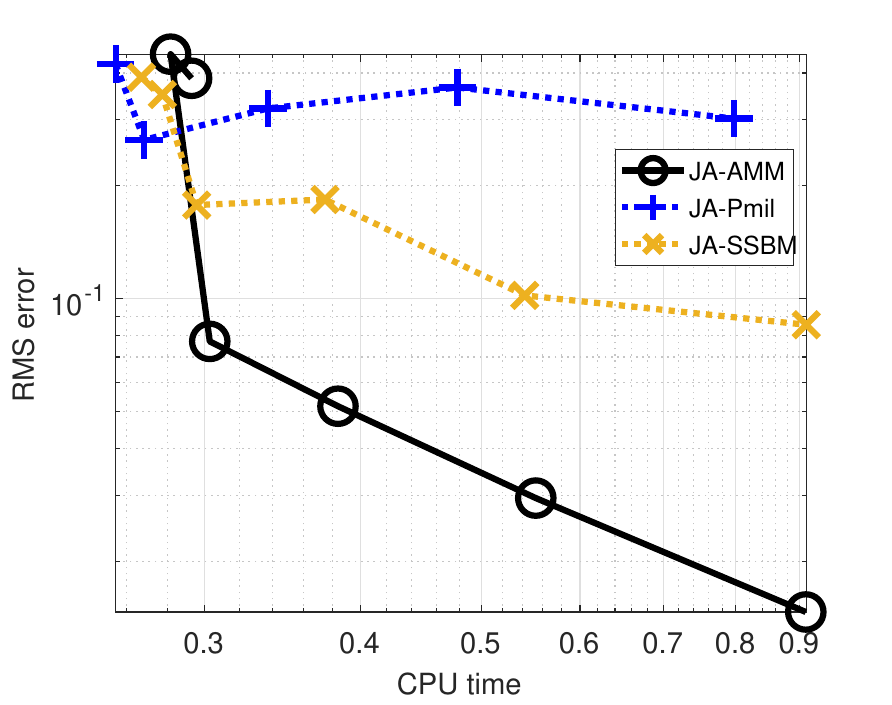}
    
    \caption{Strong convergence and efficiency of \JAAMM for approximating the one-dimensional system with multiplicative noise, when jump intensity increases. (a) and (b) with intensity 25; (c) and (d) with intensity 250.}
    \label{fig:1DLambda}
\end{figure}

\subsection{Two-dimensional test systems}
For 2D models, we consider the SDEs: 
\begin{multline}\label{eq:2D_model}
\Xjump{t}=\XjumpStart+\int_{0}^{t}F\big(\Xjump{s}\big)ds+\int_{0}^{t} G_i\big(\Xjump{s}\big)dW(s)\\
+\int_{0}^{t}\int_{Z}\gamma\big(z,\Xjump{r^-}\big)J_\nu(dz\,\times\,dr),\quad t\in[0,1],
\end{multline}
with $\XjumpStart=[7,9]^T$, $W(t)=[W_1(t),W_2(t)]^T$, where $W_1$ and $W_2$ are independent scalar Wiener processes, $\Xjump{t}=[X^{\texttt{J}}_1(t),X^{\texttt{J}}_2(t)]^T$,  $F(x)=[x_2-3x_1^3,x_1-3x_2^3]^T$, and three different diffusion terms are set as
\begin{align}
G_1(x)=\sigma\begin{pmatrix}x_1^2 & 0\\ 0 & x_2^2\end{pmatrix}, \, 
G_2(x)=\sigma\begin{pmatrix}x_2^2 & x_2^2\\ x_1^2 & x_1^2\end{pmatrix},\,
G_3(x)=\sigma\begin{pmatrix}1.5x_1^2 & x_2\\ x_2^2 & 1.5x_1\end{pmatrix}. \label{eq:num_Gs}
\end{align}
$G_1$ is an example of diagonal noise, $G_2$  commutative noise, and $G_3$ non-commutative noise. The jump coefficient is $\gamma=z\,x$ with $z\sim\mathcal{N}(0,0.01)$.

\begin{figure}
    \centering
 (a) \hspace{0.48\textwidth} (b)\\
\includegraphics[width=0.48\textwidth]{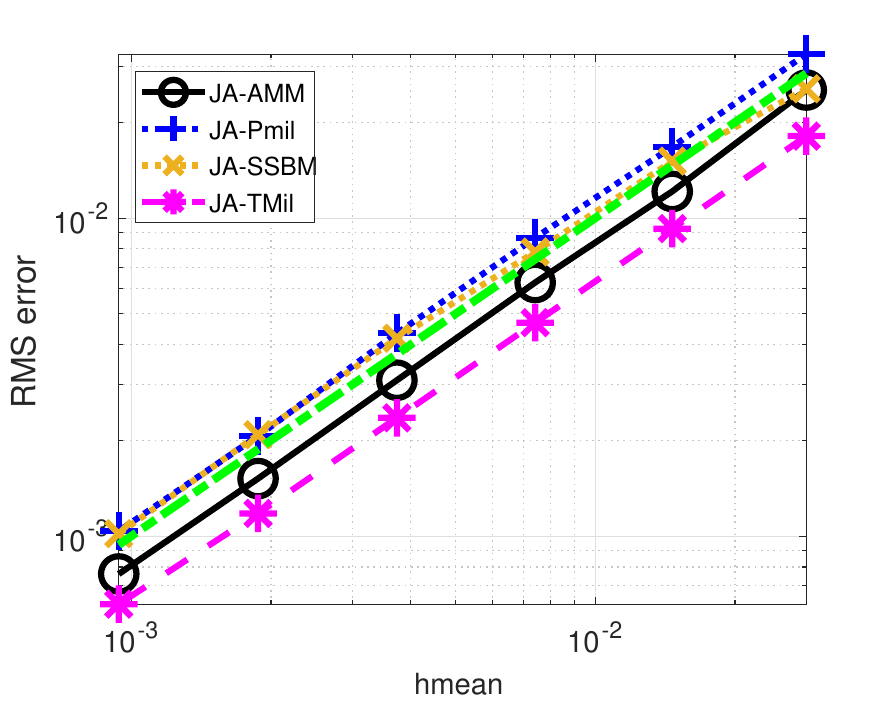}
    \includegraphics[width=0.48\textwidth]{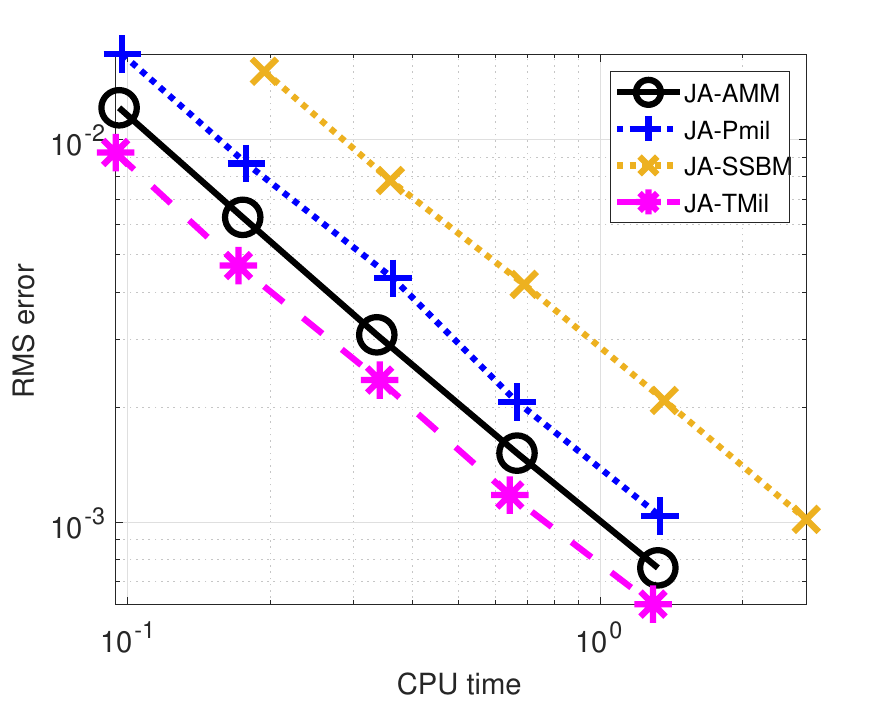}

(c) \hspace{0.48\textwidth} (d)\\
\includegraphics[width=0.48\textwidth]{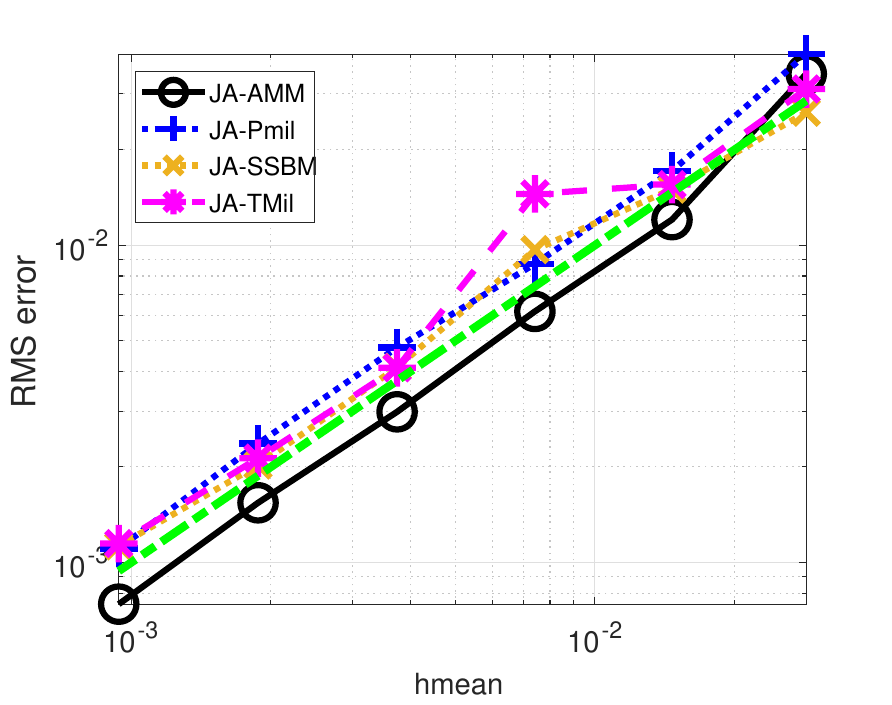}
    \includegraphics[width=0.48\textwidth]{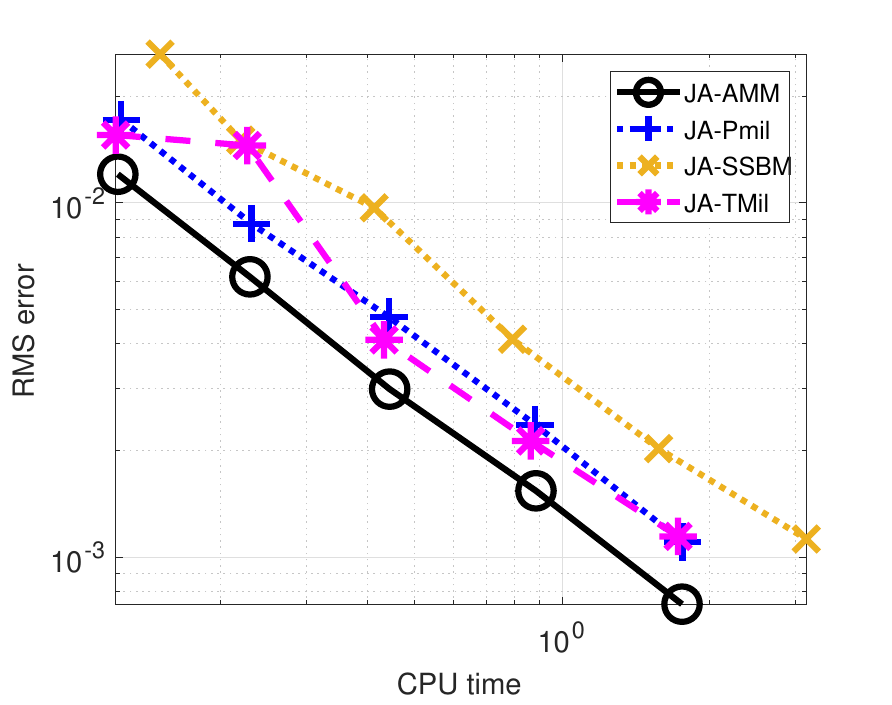}

    \caption{Strong convergence and efficiency of \JAAMM for approximating the two-dimensional system with (a) and (b) for diagonal noise; (c) and (d) for commutative noise.}\label{fig:2DLambda}
\end{figure}

For diagonal (Figure \ref{fig:2DLambda} (a)-(b)) and commutative (Figure \ref{fig:2DLambda} (c)-(d)) noises, we have initial value $[0.5,0.7]^T$, terminal time $T=1$, reference stepsize $2^{-18}$, $\hmax=[2^{-9},2^{-8},2^{-7},2^{-6},2^{-5}]$, jump intensity $\lambda=2.5$, jump sizes follow $\mathcal{N}(0,0.01)$, and $M=500$ sample trajectories. 

\begin{figure}
    \centering
    (a) \hspace{0.48\textwidth} (b)\\
    \includegraphics[width=0.49\textwidth]{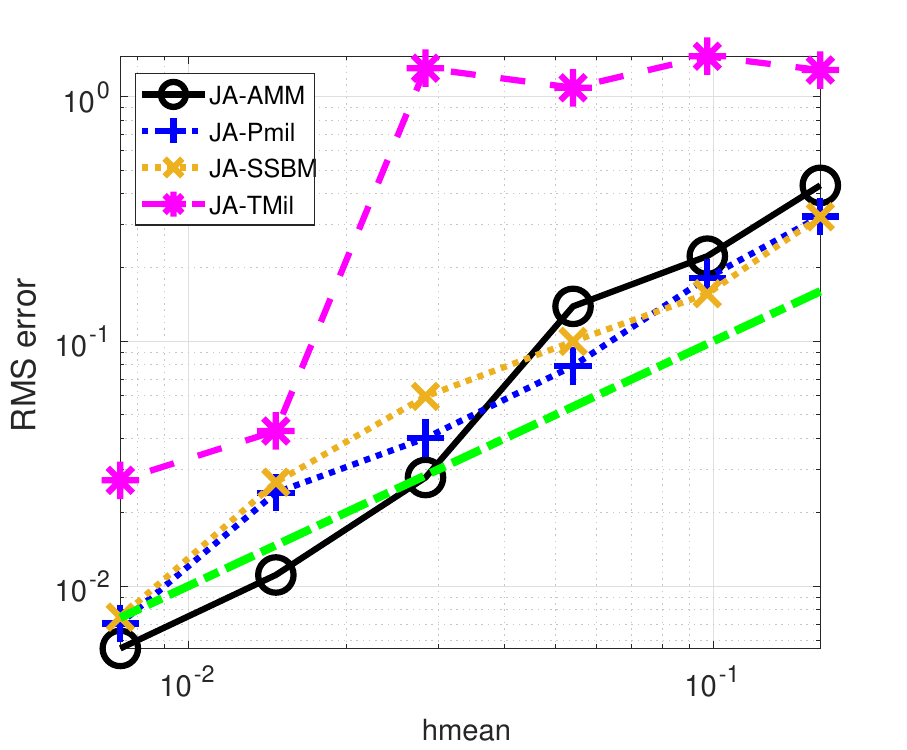}
    \includegraphics[width=0.47\textwidth]{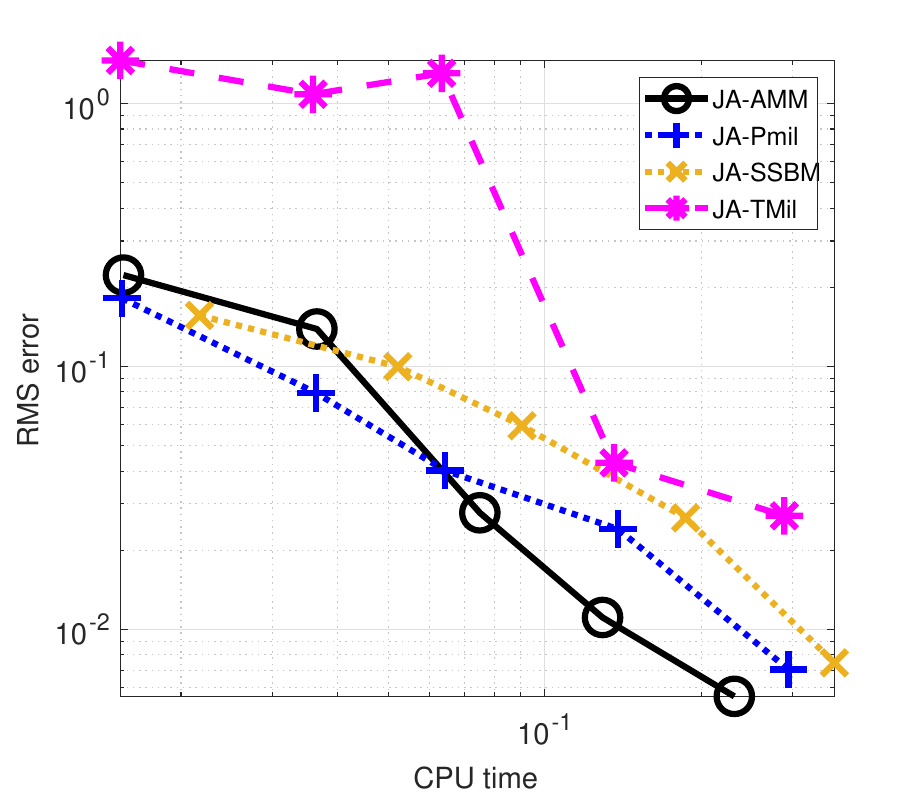}
    \caption{Strong convergence and efficiency of \JAAMM for approximating the two-dimensional system for non-commutative noise}\label{fig:2DLambdaNonCom}
\end{figure}

For non-commutative noise (Figure \ref{fig:2DLambdaNonCom}), we set $X_0=[0.5,0.7]^T$, reference stepsize $2^{-9}$, $\hmax=[2^{-6},2^{-5},2^{-4},2^{-3},2^{-2},2^{-1}]$, $M=100$, and keep other settings the same. We can see in Figure \ref{fig:2DLambdaNonCom} that \JAAMM again shows advantages in error and CPU time. 

\begin{remark}
Caution is advised when using as a backstop method a scheme such as projected or tamed Milstein, where large excursions in the solution are handled by a projection to a different part of the state space or a distortion of the SDE coefficients. These may introduce bias into the numerical solution. Nonetheless our numerical experiments illustrate the theoretical convergence rates predicted by Theorem \ref{thm:result_jump}.
\end{remark}

\section*{Acknowledgments}
The authors are grateful for the support of the RSE Saltire Research Facilitation Network Award 1832 ``Stochastic Differential Equations: Theory, Numerics, and Applications''.

\newpage
\bibliographystyle{abbrv}
\bibliography{Manuscript}

\end{document}